\documentclass[11pt,a4paper,reqno]{amsart}
\usepackage[T1]{fontenc}
\usepackage{microtype, a4wide, textcomp,fbb}
\usepackage{
    amsmath,  amssymb,  amsthm,   amscd,
    gensymb,  graphicx, etoolbox, booktabs,
    stackrel, mathtools    
}
\usepackage[usenames,dvipsnames]{xcolor}
\definecolor{darkblue}{rgb}{0,0,0.7}
\definecolor{darkred}{rgb}{0.7,0,0}
\usepackage[colorlinks=true, linkcolor=darkred, citecolor=darkblue, urlcolor=blue, pagebackref=true, breaklinks=true]{hyperref}
\usepackage[capitalise]{cleveref}
\usepackage{placeins}
\usepackage{relsize}
\usepackage{todonotes}
\usepackage{soul}
\usepackage{enumitem}
\usepackage{setspace}
\usepackage{wrapfig}
\usepackage{etoolbox}
\BeforeBeginEnvironment{wrapfigure}{\setlength{\intextsep}{0pt}}
\usepackage{tikz}
\usetikzlibrary{arrows}
\usetikzlibrary{shapes}
\usepackage{float}
\usepackage{tabularx}
\usepackage{kantlipsum}
\usepackage{array}
\usepackage{mathrsfs}
\usepackage[margin=2.5cm]{geometry}

\newtheorem{proposition}{Proposition}[section]
\newtheorem{lemma}[proposition]{Lemma}
\newtheorem{theorem}[proposition]{Theorem}

\newtheorem{corollary}[proposition]{Corollary}
\newtheorem{question}[proposition]{Question}

\theoremstyle{definition}
\newtheorem{remark}[proposition]{Remark}

\newtheorem{definition}[proposition]{Definition}

\newtheorem{innercustomthm}{Theorem}
\newenvironment{customthm}[1]
  {\renewcommand\theinnercustomthm{#1}\innercustomthm\itshape}
  {\endinnercustomthm}

\newcommand{\reg}{{\rm reg}}
\newcommand{\depth}{{\rm depth}}

\newcommand{\Fnumber}{\nu_\mathtt{F}}

\def\H{\mathcal{H}}

\advance\headheight1.15pt

\def\H{\mathcal{H}}
\def\K{\mathbb{K}}
\def\G{\mathcal{G}}
\def\C{\mathcal{C}}
\def\B{\mathcal{B}}

\def\E{\mathcal{E}}
\def\N{\mathcal{N}}

\def\reg{\mathrm{reg}}
\def\adms{\mathrm{ind}}

\def\aim{\mathrm{aim}}

\def\l{(}
\def\r{)}
\def\x{\mathbf x}

\def\Bl{\mathcal{B}_G}
\def\p{\mathfrak p}

\def\height{\mathrm{ht}}
\def\supp{\mathrm{supp}}
\def\ass{\mathrm{Ass}}
\def\sqf{\mathrm{sqf}}
\def\adm{\mathrm{adm}}
\def\ind{\alpha}
\def\F{\mathtt{F}}
\def\height{\mathrm{ht}}
\def\supp{\mathrm{supp}}
\def\ass{\mathrm{Ass}}

\def\sqf{\mathrm{sqf}}

\begin{document}

\title[Squarefree-power-like function]{Admissible set and squarefree-power-like function with applications to squarefree symbolic powers}

\author{Trung Chau}
\address{Chennai Mathematical Institute, India}
\email{chauchitrung1996@gmail.com}

\author{Kanoy Kumar Das}
\address{Chennai Mathematical Institute, India}
\email{kanoydas@cmi.ac.in; kanoydas0296@gmail.com}

\author{Amit Roy}
\address{Chennai Mathematical Institute, India}
\email{amitiisermohali493@gmail.com}

\author{Kamalesh Saha}
\address{Department of Mathematics, SRM University-AP, Amaravati 522240, Andhra Pradesh, India}
\email{kamalesh.s@srmap.edu.in; kamalesh.saha44@gmail.com}


\vspace*{-0.4cm}
\begin{abstract}
    We introduce the abstract notion of squarefree-power-like functions, which unify the sequences of squarefree ordinary and symbolic powers of squarefree monomial ideals. By employing the Tor-vanishing criteria for mixed sums of ideals, we establish sharp lower bounds for their Castelnuovo-Mumford regularity in terms of what we call the \emph{admissible set} of the associated hypergraph. As an application, we derive the first general combinatorial lower bound for the regularity of squarefree symbolic powers of monomial ideals. In the setting of edge ideals, by exploiting the special combinatorial structures of block graphs and Cohen-Macaulay chordal graphs, we show that this bound turns into an exact formula for all squarefree symbolic powers of block graphs, as well as for the second squarefree symbolic powers of edge ideals of Cohen-Macaulay chordal graphs.
\end{abstract}

\maketitle


\section{Introduction}

The study of various powers of ideals, such as ordinary powers, symbolic powers, integral closure powers, etc. has long been a central theme in commutative algebra. Much of the work in this area focuses on homological invariants of these powers and seeks to understand their asymptotic behavior. A classical result in this direction is Brodmann’s theorem \cite{Brodmann1979}, which establishes the stability of the associated primes of powers of ideals. Cutkosky, Herzog, and Trung \cite{CutkoskyHerzogTrung1999}, and independently Kodiyalam \cite{Kodiyalam1999}, proved that the regularity of powers of a homogeneous ideal eventually becomes a linear function; similar results hold in greater generality \cite{HoaTrung2007}, and in particular for symbolic powers of ideals \cite{DungHienNguyenTrung2021,HerzogHibiTrung(VertexCoverAlgebras)2007}. In contrast, the study of specific small powers of a given class of ideals is much more challenging, as many invariants tend to behave unpredictably for lower powers (see, for instance, \cite{MinhEtAl(SmallPowers)2022, MinhVu(SmallPowers)2024}).


Squarefree powers were introduced very recently in \cite{BHZN} and have since been investigated extensively by several authors (see, for example, \cite{CFL1, DRS20242, CMSimplicialForests, EHHS, FiHeHi, ficarraCM2024, KaNaQu, Fakhari2024}, and the references therein). There are two important motivations to study squarefree powers. First, they are closely connected to matching theory, providing an algebraic framework to study matchings in a graph. Secondly, the minimal free resolution of a squarefree power of an ideal appears as a subcomplex of the minimal free resolution of the corresponding ordinary power; hence, squarefree powers of an ideal provide valuable information about the ordinary powers. Interestingly, from the definition of squarefree powers, it follows that for any given squarefree monomial ideal $I$, there exists an integer $k>0$ such that $I^{[n]}=0$ for all $n \geq k$. Thus, unlike ordinary powers, square-free powers do not admit an asymptotic behavior, making the study of their homological invariants significantly different.

Recently, Fakhari introduced a symbolic analogue of this new notion in~\cite{FakhariSymbSq-free}, referring to it as the \emph{squarefree part of symbolic powers}. In this article, we refer to these simply as squarefree symbolic powers (defined in \Cref{sec: Preli}). The sequences of squarefree (ordinary) powers and squarefree symbolic powers of squarefree monomial ideals exhibit similar behaviors, for instance, both are decreasing and satisfy certain Tor-vanishing conditions (see \Cref{sec: Preli} for details). A central motivation for this work is to explore and unify these two sequences within a broader framework, with the aim of studying their Castelnuovo–Mumford regularity. This leads us to the following guiding question:


\begin{question}
Let $\mathcal{I} = \{I_i\}_{i \geq 0}$ be a decreasing sequence of squarefree monomial ideals arising from various powers of a monomial ideal in a polynomial ring. Is there an effective way to provide a combinatorial bound for the Castelnuovo–Mumford regularity for all such ideals?
\end{question}

To address this question, we introduce the notion of \emph{squarefree-power-like function} $\F$ along with an associated combinatorial invariant, which we call the  \emph{$k$-admissible $\F$-number}. Let $I$ be a squarefree monomial ideal and $\mathcal{G}(I)$ be the set of minimal monomial generators of $I$. We can consider $I$ as the edge ideal $I=I(\H)$, where $\H$ is the hypergraph corresponding to the ideal $I$. For each positive integer $k\geq 1$, we define the function $\F(\H, k)$ to be a squarefree monomial ideal, which satisfies some desired properties (see \Cref{def:sq-free power like}). This, in particular, provides a shared platform to study both the notions of squarefree powers $I^{[k]}$ and squarefree symbolic powers $I^{\{k\}}$. We further introduce the concept of a \emph{$k$-admissible $\F$-set} of $\H$ and define the associated combinatorial invariant, the \emph{$k$-admissible $\F$-number}, denoted by $\adm^{\F}(\H,k)$. This invariant plays a key role in establishing one of the main results of the paper described below.

\begin{customthm}{\ref{thm:lower-bound}}
Let $\H$ be a hypergraph and $\F$ a squarefree-power-like function. Then, for any $1 \le k \le \nu_{\F}(\H)$, we have
\[
\reg \left( \F(\H,k) \right) \geq \adm^{\F}(\H,k) + k.
\]
\end{customthm}

We note that this bound is sharp when $I(\H)$ is a complete intersection, that is, when $\H$ is a disjoint union of edges (see \Cref{rmk: CI reg}). Furthermore, we show that both squarefree powers and squarefree symbolic powers fit into this framework (see \Cref{ordinary SPLF} and \Cref{Symbolic SPLF}). In particular, the lower bound in \Cref{thm:lower-bound} recovers the bound established for squarefree powers in \cite[Theorem 3.11]{ChauDasRoySahaSqfOrd}. 

For squarefree symbolic powers $I(\H)^{\{k\}}$, we refine the concept of $k$-admissible $\F$-sets by incorporating combinatorial invariants of the underlying hypergraph. This leads to the definition of \emph{$k$-admissible independence number} of $\H$, denoted by $\adms(\H, k)$ (see \Cref{def: sym aim for graph}). It follows from the definition of squarefree symbolic powers that $I(\H)^{\{k\}}=0$ whenever $k> \height(I(\H))$ and using \Cref{thm:lower-bound} we establish that for all $1\leq k\leq \height(I(\H))$, $ \reg (I(\H)^{\{k\}}) \geq \adms(\H,k)+k$ (\Cref{cor:lower bound sqf symbolic}). It is worthwhile to mention that the only known lower bound for the regularity of squarefree symbolic powers is provided by Fakhari in \cite[Theorem 3.3]{FakhariSymbSq-free} for a graph $G$, and the power $k$ is always bounded above by the induced matching number $\nu(G)$ in this case. Note that, in general, $\nu(G)\le \height(I(G))$ for an arbitrary graph $G$. Moreover, the difference can be arbitrarily large (for instance, complete graphs). Thus, \Cref{cor:lower bound sqf symbolic} substantially extends the only known lower bound for the regularity of squarefree symbolic powers to all squarefree symbolic powers of edge ideals of arbitrary hypergraphs.


We further investigate the sharpness and compatibility of our bounds with those from \cite{ChauDasRoySahaSqfOrd, CFL1, ErHi1, FakhariSymbSq-free}. For complete graphs, the lower bound is attained (see \Cref{lem: reg of complete}). In \cite{CFL1}, Crupi, Ficarra, and Lax refined the earlier bound of \cite{ErHi1} and derived an exact formula for the regularity of squarefree powers of forests in terms of the combinatorial invariant $\aim(G,k)$. This result was later extended to block graphs in \cite{ChauDasRoySahaSqfOrd}. Since block graphs form a natural and important subclass of chordal graphs that generalizes forests, the comparison between these invariants becomes particularly meaningful. Notably, for the class of forests $I(G)^{\{k\}}=I(G)^{[k]}$ and the notion $\adms(G,k)$ coincides with $\aim(G,k)$. However, for general block graphs, neither of these identities needs to hold, that is, $I(G)^{\{k\}}\neq I(G)^{[k]}$ and $\adms(G,k) \neq \aim(G,k)$, in general. Our second main goal in this context is to find an exact formula for the regularity of squarefree symbolic powers of edge ideals of block graphs in the spirit of \cite[Theorem 4.12]{ChauDasRoySahaSqfOrd}. Interestingly, in this case, we are able to explicitly determine $\reg(I(G)^{\{k\}})$ in terms of the combinatorial invariant $\adms(G,k)$, and we do this by establishing several technical lemmas and combinatorial constructions to track $\adms(G,k)$ and the minimal generators of $I(G)^{\{k\}}$. In particular, this leads to our second main theorem, stated as follows:

\begin{customthm}{\ref{thm: block graph main}}
Let $G$ be a block graph. Then for each $1 \le k \le \height(I(G))$, we have
\[
\reg(I(G)^{\{k\}}) = \adms(G,k) + k.
\]
    
\end{customthm}

\noindent
Moreover, we show in \Cref{thm: CMchordal} that a similar formula holds for the second squarefree symbolic power of Cohen-Macaulay chordal graphs, analogous to \cite[Theorem 5.8]{ChauDasRoySahaSqfOrd}. 

This paper is organized as follows. In \Cref{sec: Preli}, we recall some preliminary notions from combinatorics and commutative algebra. We also review homological properties of filtrations of monomial ideals and study how they can be adapted for the case of squarefree powers. The main section of this article is \Cref{sec: main}, where we introduce the abstract notion of squarefree-power-like function and establish a general lower bound for regularity. As an application, in \Cref{sec: app to sym pwr}, we relate this general lower bound to certain well-known graph-theoretic invariants, which yields a sharp combinatorial lower bound for the squarefree symbolic powers of monomial ideals. Furthermore, in \Cref{sec: block} and \ref{sec: CM chordal}, we show that this general lower bound is attained for all squarefree symbolic powers of block graphs, and for the second squarefree symbolic powers of Cohen-Macaulay chordal graphs, respectively.

\section{Preliminaries and basic results}\label{sec: Preli}

In this section, we begin by recalling some preliminary notions from combinatorics and commutative algebra, and we refer to \cite{BM76,Edmonds,RHV} for any undefined terminology. After setting up the necessary background, we proceed to derive results concerning the regularity of the mixed sum of eventually vanishing decreasing sequences of monomial ideals. These results will play a central role in the latter part of the article.
	
	\subsection{Hypergraph}
	
	A (simple) hypergraph $\H$ is an ordered pair $\H=(V(\H),E(\H))$, where $V(\H)$ is a finite set, called the \textit{vertex set}, and $E(\H)$ is a family of subsets of $V(\H)$, called the \textit{edge set}, such that no two elements contain one another. The elements of $V(\H)$ are referred to as \textit{vertices}, and the elements of $E(\H)$ as \textit{edges}. For $W\subseteq V(\H)$, the \textit{induced sub-hypergraph} of $\H$ on $W$, denoted by $\H[W]$, is defined by $
	V(\H[W])=W \, \text{and} \, E(\H[W])=\{\E\in E(\H)\mid \E \subseteq W\}$. The hypergraph $\H[V(\H)\setminus W]$ is also denoted by $\H\setminus W$. If $W=\{x\}$ for some $x\in V(\H)$, then we simply write $\H\setminus x$. A \textit{matching} of $\H$ is a collection of edges $M\subseteq E(\H)$ such that $\E\cap \E'=\emptyset$ for any distinct $\E,\E'\in M$. A matching $M$ is called an \textit{induced matching} if $
	E\big(\H\big[\cup_{\E\in M}\E\big]\big)=M$.
	The \textit{induced matching number} of $\H$, denoted by $\nu_1(\H)$, is the maximum cardinality of an induced matching in $\H$.
	
	A (simple) \textit{graph} $G$ is a hypergraph whose edges all have cardinality two.  
	For $A\subseteq V(G)$, the set of \textit{closed neighbors} of $A$ is defined by $
	N_{G}[A]\coloneqq A\cup \{x\in V(G)\mid \{x,y\}\in E(G)\ \text{for some } y\in A\}$. The set of \textit{open neighbors} of $A$ is $
	N_{G}(A)\coloneqq N_{G}[A]\setminus A$.
	If $A=\{x\}$, then we simply write $N_G(x)$ for $N_G(\{x\})$. For $x\in V(G)$, the number $|N_G(x)|$ is called the \textit{degree} of $x$ and is denoted by $\deg(x)$. If $\deg(x)=1$, then the unique edge adjacent to $x$ is called a whisker of $G$. Given two vertices $x,y\in V(G)$, an \textit{induced path of length $n$} between $x$ and $y$ is a sequence of vertices $P: a_1,\ldots,a_n$ such that $a_1=x$, $a_n=y$, and $\{a_i,a_j\}\in E(G)$ if and only if $j=i-1$ or $j=i+1$. A \textit{complete graph} on $n$ vertices, denoted $K_n$, is a graph in which every pair of vertices is joined by an edge. A \textit{clique} of a graph $G$ is a complete induced subgraph of $G$. A maximal clique in $G$ is called a {\it block}. A vertex $x\in V(G)$ is called a \textit{free vertex} (or a \textit{simplicial vertex}) of $G$ if $G[N_{G}(x)]$ is complete. Two edges $e,e'\in E(G)$ are said to form a \textit{gap} in $G$ if $\{e,e'\}$ is an induced matching in $G$. The \textit{complement} of a graph $G$, denoted $G^c$, is the graph with $V(G^c)=V(G) \, \text{and}\, 
	E(G^c)=\{\{x,y\}\mid \{x,y\}\notin E(G)\}$.
	
	A \textit{cycle} of length $m$, denoted $C_m$, is the graph with $V(C_m)=\{x_1,\ldots,x_m\}, \, 
	E(C_m)=\{\{x_1,x_m\}\}\cup\{\{x_i,x_{i+1}\}\mid 1\le i\le m-1\}$. The cycle $C_3$ is also sometimes called a {\it triangle}. A graph $G$ is called a \textit{forest} if it contains no cycle. A connected forest is called a \textit{tree}. A graph $G$ is called \textit{chordal} if it has no induced cycle $C_n$ with $n\geq 4$. A graph is said to be \textit{weakly chordal} if both $G$ and $G^c$ contain no induced cycle $C_n$ with $n\geq 5$. Note that every chordal graph is weakly chordal. A \emph{block graph} $G$ is a chordal graph such that any two blocks in $G$ intersect in at most one common vertex. A graph $G$ is said to be a whisker graph if every vertex of $G$ is either of degree $1$ or is a neighbor of a degree $1$ vertex.
		
	\subsection{Free resolution and square-free power} For a graded ideal $I\subseteq R$, the graded minimal free resolution of $I$ is an exact sequence
	\[
	\mathcal F_{\cdot}: \quad 
	0 \longrightarrow F_t \xrightarrow{\partial_{t}} \cdots 
	\xrightarrow{\partial_{2}} F_1 \xrightarrow{\partial_1} F_0 
	\xrightarrow{\partial_0} I \longrightarrow 0,
	\]
	where $F_i=\bigoplus_{j\in\mathbb N} R(-j)^{\beta_{i,j}(I)}$ is a free $R$-module for each $i\ge 0$, and $R(-j)$ is the polynomial ring $R$ with grading shifted by $j$.  
	The integers $\beta_{i,j}(I)$ are called the $i^{\text{th}}$ $\mathbb N$-graded \textit{Betti numbers} of $I$ in degree $j$. The \textit{Castelnuovo--Mumford regularity} (or simply, \textit{regularity}) of $I$, denoted by $\reg(I)$, is defined as $\reg(I)\coloneqq \max\{\, j-i \mid \beta_{i,j}(I)\neq 0 \,\}$. If $I$ is the zero ideal, we adopt the convention $\reg(I)=0$. Moreover, we define $\reg(R)\coloneqq 0$. The following two well-known results on the regularity of graded ideals will be used frequently in subsequent sections.
	
	\begin{lemma}{\cite[Lemma~2.10]{DHS}}\label{regularity lemma}
		Let $I \subseteq R=\mathbb K[x_1,\ldots,x_n]$ be a monomial ideal, $m$ a monomial of degree $d$ in $R$, and $x$ an indeterminate in $R$. Then:
		\begin{enumerate}[label=(\roman*)]
			\item $\reg(I+( x))\le \reg(I)$,
			\item $\reg(I:x)\le \reg(I)$,
			\item $\reg(I) \le \max\{\reg(I:m) + d,\ \reg(I+( m))\}$.
		\end{enumerate}
	\end{lemma}
	
	\begin{lemma}{\cite{RHV}}\label{reg sum}
		Let $I_1\subseteq R_1=\mathbb K[x_1,\ldots,x_m]$ and $I_2\subseteq R_2=\mathbb K[x_{m+1},\ldots,x_n]$ be graded ideals.  
		Consider the ideal $I=I_1R+I_2R \subseteq R=\mathbb K[x_1,\ldots,x_n]$. Then
		\[
		\reg(I)=\reg(I_1)+\reg(I_2)-1.
		\]
	\end{lemma}

	Let $\H$ be a hypergraph on the vertex set $V(\H)=\{x_1,\ldots,x_n\}$. Identifying the vertices of $\H$ as indeterminates, we consider the polynomial ring $R=\mathbb K[x_1,\ldots,x_n]$ over a field $\mathbb K$. Corresponding to any $F\subseteq \{x_1,\ldots,x_n\}$, we denote the square-free monomial $\prod_{x_i\in F}x_i$ of $R$ by $\x_{F}$. The edge ideal of $\H$, denoted by $I(\H)$, is a square-free monomial ideal of $R$ defined as 
	\[ I(\H)\coloneqq \left ( \x_{\E}\mid \E\in E(\H)\right ). \]
	Note that any square-free monomial ideal can be seen as the edge ideal of a hypergraph. In general, for a monomial ideal $I$, we denote by $\mathrm{Ass}_I$ the set of its minimal prime ideals.  
Given an integer $k \geq 1$, the {\it $s^{\text{th}}$ symbolic power} of $I$, written as $I^{(k)}$, is defined by
\[
I^{(k)} \coloneqq 
\bigcap\limits_{\mathfrak p \in \mathrm{Ass}_I} 
\ker\!\left(R \longrightarrow (R/I^k)_{\mathfrak p}\right).
\]
Suppose now that $I$ is a squarefree monomial ideal in $R$, and let $
I = \mathfrak p_1 \cap \cdots \cap \mathfrak p_r$ be its irredundant primary decomposition, where each $\mathfrak p_i$ is generated by a subset of the variables of $S$.  
Then, by \cite[Proposition 1.4.4]{HHBook}, for every $k \geq 1$ we have
\[
I^{(k)} = \mathfrak p_1^k \cap \cdots \cap \mathfrak p_r^k.
\]
The ideal generated by the square-free monomials in $\G(I^{(k})$ is denoted by $I^{\{k\}}$, and is called as the {\it $k^{\text{th}}$ squarefree symbolic power} of $I$. Thus, by definition, we obtain the following.
\begin{lemma}\textup{(cf. \cite{FakhariSymbSq-free})}\label{sqfSymbolic}
    Let $I\subseteq R$ be a squarefree monomial ideal, and $I=P_1\cap\cdots \cap P_r$ be an irredundant primary decomposition, where every $P_i$ is generated by a subset of the variables in $R$. Then 
    \[
    I^{\{k\}}=P_1^{\{k\}}\cap \cdots \cap P_r^{\{k\}}.
    \]
\end{lemma}

\noindent
By convention, $I^{\{k\}}=R$ for all $k\le 0$. The following two lemmas on the squarefree symbolic powers of edge ideals of graphs will be heavily used in the subsequent sections.

\begin{lemma}\cite[Lemma 4.1]{FakhariSymbSq-free}\label{lem: symb sq-free 1}
    Let $G$ be a graph and $x_1$ is a simplicial vertex of $G$ with $N_G[x_1]=\{x_1,\ldots , x_d\}$, for some integer $d\geq 2$. Then for every integer $k\geq 1$,
    \[
    (I(G)^{\{k\}}:x_1x_2\cdots x_d)=I(G\setminus N_G[x_1])^{\{k-d+1\}}.
    \]
\end{lemma}

\begin{lemma}\cite[Lemma 4.2]{FakhariSymbSq-free}\label{lem: symb sq-free 2}
    Let $G$ be a graph and suppose that $W = \{x_1,\ldots , x_d\}$ is a nonempty subset of vertices of $G$. Then for every integer $k \geq 1$,
    \[
    \reg(I(G)^{\{k\}}:x_1)\leq \max \{\reg (I(G\setminus U_1)^{\{k\}}: \mathbf{x}_{U_2})+|U_2|-1\mid x_1\in U_2, U_1\cap U_2=\emptyset, U_1\cup U_2=W\}.
    \]
\end{lemma}

\subsection{Mixed sum of monomial ideals}

In this subsection, we consider the regularity of the mixed sum of a filtration of ideals in a polynomial ring. A sequence $\mathcal{I}=\{I_i\}_{i\geq 0}$ of ideals in a polynomial ring $A=\K[x_1,\dots, x_r]$ is said to be a {\it filtration} if it satisfies the following three conditions:
	\begin{enumerate}
		\item[(i)] $ I_0=A$;
		
		\item[(ii)] $ I_1$ is a non-zero proper ideal in $A$;
		
		\item[(iii)] $ I_i\supseteq I_{i+1}$ for all $i\ge 0$.
	\end{enumerate} 
	\color{black}
	
	Let $\mathcal{I}=\{I_i\}_{i\geq 0}$ (resp, $\mathcal{J}=\{J_j\}_{j\geq 0}$) be a filtration of monomial ideals of $A=\K[x_1,\dots, x_r]$ (resp, $B=\K[y_1,\dots, y_s]$). For each $n\geq 0$,  
\[
Q_n=\sum_{i+j=n}(I_iR)(J_jR),
\]
is an ideal of $R=A\otimes_\K B$.
The ideal $Q_n$ is sometimes referred to as a \emph{mixed sum} of $\mathcal{I}$ and $\mathcal{J}$. Homological invariants of $Q_n$ are computable directly from those of ideals of $\mathcal{I}$ and $\mathcal{J}$ in special cases.
Recall that $\mathcal{I}$ is called \emph{Tor-vanishing} if for each $i\geq 1$, the inclusion map $I_i\to I_{i-1}$ is Tor-vanishing, i.e., the map $\operatorname{Tor}_n^A(\K, I_i)\to \operatorname{Tor}_n^A(\K, I_{i-1})$ is the zero map for all $n\geq 0$. 

\begin{theorem}\protect{\cite[Theorem 5.3]{THT2020}}\label{thm:reg-Tor-filtration}
    If $\mathcal{I}$ and $\mathcal{J}$ are Tor-vanishing, then 
    \[
    \reg (Q_n) = \max_{\substack{i\in [1,n-1]\\
    j\in [1,n]}} \{ \reg (I_{n-i}) + \reg (J_i) , \quad  \reg (I_{n-j+1}) + \reg (J_j) - 1  \}
    \]
    for any $n\geq 0$.
\end{theorem}

In the sequel, we often encounter filtration that is eventually $0$. For this reason, for a filtration of monomial ideal $\mathcal{I}=\{I_i\}_{\geq 0}$, we define
\[
\nu (\mathcal{I}) = \sup \{n\geq 0\colon I_n\neq 0 \}.
\]
With this invariant one can optimize \cref{thm:reg-Tor-filtration}, as for any $n\geq 0$, we have
\[
Q_n = \sum_{ \max \{ 0,n-\nu(\mathcal{I})   \}  \leq i\leq \min \{n,\nu(\mathcal{J})\}} I_{n-i} J_{i}.
\]
As the arguments are almost line-by-line as those of the proof of \cite[Theorem 5.3]{THT2020}, we state our result and give only a sketch of the proof.

\begin{theorem}\label{thm:reg-Tor-filtration-2}
	If $\mathcal{I}$ and $\mathcal{J}$ are Tor-vanishing, then
	\begin{align*}
		\reg (Q_n)  &=\begin{multlined}[t]
			\max_{\substack{i\in [\max\{0,n-\nu(\mathcal{I)}\},\ \min\{n,\nu(\mathcal{J})]\\j\in [\max\{0,n-\nu(\mathcal{I)}\}+1,\ \min\{n,\nu(\mathcal{J})]}}   \{\reg (I_{n-i}) + \reg (J_i) , \quad  \reg( I_{n-j+1}) + \reg (J_j) -1  \}
		\end{multlined}    \\
        &=\begin{multlined}[t]
			\max_{\substack{i\in [\max\{1,n-\nu(\mathcal{I})\},\ \min\{n-1,\nu(\mathcal{J})\} ]\\j\in [\max\{0,n-\nu(\mathcal{I)}\}+1,\ \min\{n,\nu(\mathcal{J})]}}   \{\reg (I_{n-i}) + \reg (J_i) , \quad  \reg( I_{n-j+1}) + \reg (J_j) -1  \}
		\end{multlined}    
	\end{align*}
	for any $n\in [0,\nu(\mathcal{I})+\nu(\mathcal{J})]$.
\end{theorem}

\begin{proof}[Sketch of the proof]
    Set
    \[
	a\coloneqq\max\{0,n-\nu(\mathcal{I)}\}, \quad b= \min\{n,\nu(\mathcal{J})\}, \quad a'\coloneqq \max\{1,n-\nu(\mathcal{I})\}, \quad \text{and}\quad b'\coloneqq \min\{n-1,\nu(\mathcal{J})\}.  
	\]
    For each $t\in [a,b]$, set
    \[
    P_{n,t} = \sum_{ i =a }^t I_{n-i} J_{i}.
    \]
    Then $P_{n,t}=P_{n,t-1}+I_{n-t}J_{t}$ is a Betti splitting (cf. \cite[Definition~1.1]{FranciscoHaVanTuyl2009}), due to the fact that $\mathcal{I}$ and $\mathcal{J}$ are Tor-vanishing filtration of monomial ideals generated in distinct sets of variables, and $P_{n,t-1}\cap I_{n-t}J_{t}=I_{n-t+1}J_{t}$. Therefore, we have
    \[
    \reg P_{n,t} = \max\{ \reg P_{n,t-1}, \quad \reg I_{n-t} + \reg J_{t}, \quad \reg I_{n-t+1} + \reg J_{t}-1 \}.
    \]
    Using induction on $t$, remarking that $Q_n=P_{n,b}$, with the base case $\reg P_{n,a} =\reg (I_{n-a}J_{a}) = \reg I_{n-a}+ \reg J_{a}$, we obtain
    \begin{align*}
		\reg \ Q_n = \begin{multlined}[t]
			\max_{\substack{i\in [a,b]\\j\in [a+1,b]}}   \{\reg \ I_{n-i} + \reg \ J_i , \quad  \reg \ I_{n-j+1} + \reg \ J_j -1  \},
		\end{multlined}    
	\end{align*}
    which proves the first equality. There are potentially a few numbers in the above set that can be ignored without changing the maximum. We have
    \begin{align*}
        \reg \ I_0 + \reg \ J_n &= 0 + \reg \ J_n \leq \reg \ I_{1} + \reg\  J_n -1,\\
        \reg \ I_n + \reg \ J_0 &= \reg  \ I_n + 0 \leq \reg \ I_{n} + \reg\  J_1 -1
    \end{align*}
    as $I_1$ and $J_1$ are non-zero proper ideals. Thus, we can put new bounds on $i$ without changing the equality, as follows:
    \begin{align*}
        i&\geq \max \{a,1\} = \max\{0,n-\nu(\mathcal{I}),1\} = a',\\
        i&\leq \min \{ b,n-1 \} = \min \{n, \nu(\mathcal{J}), n-1\} = b'.
    \end{align*}
    The second equality then follows.
\end{proof}

We recall a sufficient condition for an inclusion of monomial ideals to be Tor-vanishing. For a monomial $m$, let $\supp(m)$, called the \emph{support} of $m$, denote the set of variables that divide $m$. For a monomial ideal $I$, let  $\partial^*(I)$ denote the monomial ideal generated by $m/x$, where $m\in \G(I)$ and $x\in \supp(m)$.

\begin{lemma}\protect{\cite[Lemma 4.2 and Proposition 4.4]{HV19}}\label{lem:del-condition}
    If $I\subseteq  I'$ are two nonzero monomial ideals such that $\partial^*(I)\subseteq I'$, then the inclusion $I\hookrightarrow I'$ is Tor-vanishing.
\end{lemma}

For a monomial ideal $I$, the ordinary powers $\{I^n\}_{n\geq 0}$ and symbolic powers $\{I^{(n)}\}_{n\geq 0}$ are among known examples that satisfy the condition in \cref{lem:del-condition} (\cite[Theorem~4.5]{HV19} and \cite[Proposition~5.10]{THT2020}), and thus in particular, are Tor-vanishing. We will show a new way to obtain Tor-vanishing filtration. For a monomial ideal $I$, let $\sqf(I)$ be the ideal generated by the squarefree monomials belonging to $I$; if $I$ does not contain any squarefree monomial, then $\sqf(I)=( 0 )$. By convention, $\sqf(R)=R$.

\begin{lemma}\label{lem:sqf-del-condition}
    Let $\{I_i\}_{i\geq 0}$ be a filtration of monomial ideal and $k\geq 1$ a positive integer. If $\partial^*(I_k)\subseteq I_{k-1}$, then $\partial^*(\sqf(I_k))\subseteq \sqf(I_{k-1})$.
\end{lemma}

\begin{proof}
    We have
    \[
    \partial^*(\sqf(I_k))\subseteq \partial^*(I_k)\subseteq I_{k-1}. 
    \]
    The result then follows from taking the squarefree part of both sides, remarking that $\sqf\left(\partial^*(\sqf(I_k))\right)=\partial^*(\sqf(I_k))$ as the latter is squarefree.
\end{proof}

Two filtration of monomial ideals that are known to satisfy the condition in \cref{lem:sqf-del-condition} are $\{I^n\}_{n\geq 0}$ and $\{I^{(n)}\}_{n\geq 0}$ where $I$ is a monomial ideal (\cite[Lemma 4.3]{HV19} and \cite[Proposition 5.10]{THT2020}). Now assume that $I$ is a squarefree monomial ideal. Then by taking the squarefree part, we obtain the two filtrations $\{I^{[n]}\}_{n\geq 0}$ and $\{I^{\{n\}}\}_{n\geq 0}$. Note that $\nu\left(\{I^{[n]}\}_{n\geq 0}\right) $ is exactly the matching number of the associated hypergraph, and $\nu\left(\{I^{\{n\}}\}_{n\geq 0}\right) $ is exactly $\height(I)$. 
The following  follows directly from \cref{lem:del-condition} and \cref{lem:sqf-del-condition}.

\begin{corollary}\label{cor:squarefree-powers-Tor-vanishing}
    Let $I$ be a squarefree monomial ideal. Then $\{I^{[n]}\}_{n\geq 0}$ and $\{I^{\{n\}}\}_{n\geq 0}$ are Tor-vanishing. 
\end{corollary}

Before going to the next section, we set up some notations and terminologies that we will use in the subsequent sections. For $A\subseteq \{x_1,\ldots,x_n\}$, $(A)$ denote the monomial ideal $(x\mid x\in A)$ in the polynomial ring $R=\mathbb K[x_1,\ldots,x_n]$. If $\p$ is a prime ideal in $R$ generated by monomials $\{x_i\mid x_i\in B\subseteq\{x_1,\ldots,x_n\}\}$, then we sometimes identify the set $B$ with the ideal $\p$, and write $\p$ as $\p_B$.

\section{Squarefree-power-like function, admissible set and regularity}\label{sec: main}

In this section, we define the notion of a squarefree-power-like function. The idea is to create a common framework for squarefree powers and squarefree symbolic powers of squarefree monomial ideals. The goal is to give a sharp combinatorial lower bound on their regularity.

For a monomial ideal $I$ and a monomial $m$, let $I^{\leq m}$ denote the monomial ideal generated by monomials in $I$ that divide $m$. This action is often referred to as \emph{restriction} (with respect to $m$). 

\begin{definition}[Squarefree-power-like function]\label{def:sq-free power like} Fix a set of variables $V$.
Let $\mathtt{F}$ be a function that maps a pair $(\H,k)$ of a hypergraph $\H$ with $V(\H)\subseteq V$ and a non-negative integer $k$ to a squarefree monomial ideal $\mathtt{F}(\mathcal{H},k)$ in $S=\K[V]$, satisfying the following property: 

\begin{enumerate}
\item[(a)] $\mathtt{F}(\H,0)=S$, $\mathtt{F}(\H,1)=I(\H)$, and $\mathtt{F}(\mathcal{H},k)$ is generated by monomials in $\K[V(\mathcal{H})]$.

\item[(b)] $\partial^{\ast}\,\mathtt{F}(\H,k)\subseteq \mathtt{F}(\H,k-1)$ for any positive integer $k$.


    \item[(c)] For any hypergraph $\H$, an induced sub-hypergraph $\H_1$ of $\H$, and any integer $k$, we have 
    \[
    \mathtt{F}(\H,k)^{\leq\x_{V(\H_1)}}  = \mathtt{F}(\H_1,k).
    \] 
    \item[(d)] For any hypergraphs $\H_1$ and $\H_2$ in disjoint sets of vertices, and any integer $k\geq 1$, we have \[
    \mathtt{F}(\H_1+ \H_2,k) = \sum_{i=0}^k \mathtt{F}(\H_1,i)\mathtt{F}(\H_2,k-i),
    \]
    where $\H_1+\H_2$ denotes the hypergraph whose vertex and edge sets are exactly the union of the vertex and edge sets of $\H_1$ and $\H_2$, respectively.
\end{enumerate} 
We call this function $\mathtt{F}$ a \textit{squarefree-power-like} function (on $V$). 
\end{definition}

 For the remainder of the section, $\mathtt{F}$ denotes a squarefree-power-like function. We remark that there are a few properties that can be further deduced from our definition. For a monomial ideal $I$, set $\delta(I)\coloneqq \inf\{\deg(m)\colon m\in I\}$. For a hypergraph $\mathcal{H}$, the sequence $\{\mathtt{F}(\H,i)\}_{i\geq 0}$ is a filtration of monomial ideals, and it eventually terminates at the zero ideal. Indeed, if $m$ is a monomial in $\mathtt{F}(\mathcal{H},n)$ for some integer $n$, then $m/x\in \partial^* \mathtt{F}(\mathcal{H},n)\subseteq \mathtt{F}(\mathcal{H},n-1)$. This implies that $\delta(\mathtt{F}(\mathcal{H},n))>\delta(\mathtt{F}(\mathcal{H},n-1))$ for any $n\geq 1$. Moreover, since $\mathtt{F}(\mathcal{H},n)$ is an ideal generated by squarefree monomials in $\K[V(\H)]$ (\cref{def:sq-free power like}(a)), we have $\delta(\mathtt{F}(\mathcal{H},n))\leq |V(\mathcal{H})|$ if $\mathtt{F}(\H,n)\neq 0$. Therefore, $\delta(\mathtt{F}(\mathcal{H},n))=\infty$, or equivalently, $\mathtt{F}(\mathcal{H},n)=0$,  for any $n> |V(\mathcal{H})|$. The smallest $k_0$ for which $\mathtt{F}(\H,k)=0$ for all $k>k_0$ is said to be the \emph{$\mathtt{F}$-number} of $\H$, and we will denote it by $\Fnumber(H)$. We obtain some basic results on $\mathtt{F}$-numbers below.

\begin{lemma}\label{lem:Fnumber-properties}
    \begin{enumerate}
        \item For any two hypergraphs $\H_1$ and $\H_2$ in disjoint set of vertices,
\[
    \Fnumber(\H_1\sqcup \H_2)=\Fnumber(\H_1)+\Fnumber(\H_2).
\]
        \item For any hypergraph $\H$, if $\mathtt{F}(\H,k)=\left(\prod_{x\in V(\H)} x \right)$, then $\Fnumber(\H)=k$.
    \end{enumerate}
\end{lemma}

\begin{proof}
    (1) follows immediately from \Cref{def:sq-free power like}(d). As for (2), assume that $\mathtt{F}(\H,k+1)\neq 0$. Then 
    \[
    |V(\H)|=\delta(\mathtt{F}(\H,k))< \delta(\mathtt{F}(\H,k+1)) \leq |V(\H)|,
    \]
    a contradiction. Thus, $\mathtt{F}(\H,k+1)=0$. Since $\mathtt{F}(\H,k)\neq 0$, the result follows.
\end{proof}
    
The next result justifies our name choice.

\begin{proposition}\label{prop:coincide-sqf-power}
    If $\H$ is a disjoint union of edges, then $\F(\H,k)=I(\H)^{[k]}$ for any integer $k$.
\end{proposition}

\begin{proof}
    Set $\H=\H_1+\cdots+\H_q$ where the hypergraphs $\H_1,\dots, \H_q$ have one edge each. We will prove our statement by induction on $q$.
    
    If $q=1$ then by \Cref{def:sq-free power like}(b) and \cref{lem:Fnumber-properties} (2), $\F(\H,k)$ is nonzero if and only if $k\leq 1$. The result then follows from \cref{def:sq-free power like} (a). Now assume $q\geq 2$. By induction, we can assume that \[
    \F(\H_1+\cdots+\H_{q'},k)=I(\H_1+\cdots+\H_{q'})^{[k]}\]
    whenever $q'<q$. By \cref{lem:Fnumber-properties} (2) and \cref{def:sq-free power like} (a), we have $\Fnumber(\H_i)=1$ for any $i\in [q]$. Using \cref{def:sq-free power like} (d), we obtain
    \begin{align*}
        \F(\H,k) &= \sum_{i=0}^k \F(\H_1+\cdots + \H_{q-1},i)\F(\H_q,k-i) \\
        &= I(\H_1+\cdots + \H_{q-1})^{[k]} + I(\H_1+\cdots + \H_{q-1})^{[k-1]}I(\H_q)\\
        &= I(\H_1+\cdots + \H_{q})^{[k]},
    \end{align*}
    as desired.
\end{proof}

Next, we introduce an invariant that will help us bound the regularity of $\mathtt{F}(\H,k)$.

\begin{definition}\label{def:aim-F}
    Let  $\H$ be a hypergraph. Then for any $1\leq k\leq \Fnumber(\H)$, a vertex set $C\subseteq V(\H)$ is called a \emph{$k$-admissible $\mathtt{F}$-set} of $\H$ if there is a partition $C=\sqcup_{i=1}^rC_i$ such that the following holds:
    \begin{enumerate}
        \item for any $i\in [r]$, we have $E(\H[C_i])\neq \emptyset$;
        \item for any edge $\E$ of $\H[C]$, we have $\E\in E(\H[C_i])$ for some $i\in [r]$;
        \item $k\leq \sum_{i=1}^r \Fnumber(\H[C_i]) \leq r+k-1$;
        \item $\mathtt{F}(\H[C_i], \Fnumber(\H[C_i]))$ equals the principal ideal $\left(\x_{C_i} \right)$ for any $i\in [r]$.
    \end{enumerate}

    For each such $C$, we compute the value of $|C|-\sum_{i=1}^r \Fnumber(\H[C_i])$. We call the largest value among those the \emph{$k$-admissible $\mathtt{F}$-number} of $\H$, and denote it by $\adm^{\mathtt{F}}(\H,k)$.
\end{definition}

\begin{proposition}\label{prop:existence}
    Let  $\H$ be a hypergraph. Then for any $1\leq k\leq \Fnumber(\H)$, there exists a $k$-admissible $\mathtt{F}$-set of $\H$. Consequently, $\adm^{\mathtt{F}}(\H,k)$ is well-defined for any hypergraph $\H$ and any integer $1\leq k\leq \Fnumber(\H)$.
\end{proposition}
\begin{proof}
    As $1\leq k\leq \nu_\mathtt{F}(\H)$, the ideal $\mathtt{F}(\H,k)$ is a nonzero proper squarefree monomial ideal. Consider a minimal squarefree generator $m$ of $\mathtt{F}(\H,k)$, and set $C\coloneq \supp(m)$. Using \Cref{def:sq-free power like}(a),(b), we see that conditions (1)-(2) of Definition~\ref{def:aim-F} are satisfied for the partition $C=C_1$. Moreover, by \Cref{def:sq-free power like}(c),
    \[
    \mathtt{F}(\H[C],k) = \mathtt{F}(\H,k)^{\leq m} = (m) = \left(\prod_{x\in \supp(m)}x\right).\]
    In particular, this implies that Condition (4) of Definition~\ref{def:aim-F} holds, and that $\Fnumber(\H[C])=k$, which implies Condition (3) of Definition~\ref{def:aim-F}. To sum up, $C=C_1$ is indeed a $k$-admissible $\mathtt{F}$-set of $\H$.
\end{proof}

Assume that $\H$ is a hypergraph, $1\leq k\leq \Fnumber(\H)$ an integer, and $C\subseteq V(\H)$ a $k$-admissable $\mathtt{F}$-set of $\H$. By \cref{def:sq-free power like} (c), we have $\mathtt{F}(\H,k)^{\leq \x_C} = \mathtt{F}(\H[C],k)$. By the Restriction Lemma (\cite[Lemma~4.4]{HHZ04}), $\reg\, \mathtt{F}(\H[C],k)$ is a lower bound of $\reg\, \mathtt{F}(\H,k)$. Computationally, even with the many conditions on $C$,  the number $\reg\, \mathtt{F}(\H[C],k)$ is still not easy to determine. The number  $|C|-\sum_{i=1}^r \Fnumber(\H[C_i])$, on the other hand, is easier to compute, and it turns out that this number is a lower bound for $\reg\, \mathtt{F}(\H[C],k)$. This is the main result of this section. Before proving it, we need a weaker version of \cref{thm:reg-Tor-filtration-2}.

\begin{theorem}\label{thm:refined-betti-splittings}
    Given two hypergraphs $\H_1$ and $\H_2$ with disjoint vertex sets, and a positive integer $\Fnumber(\H_1)\leq k\leq \nu_\mathtt{F}(\H_1+\H_2)$, we have
    \begin{align*}
        \reg \, \mathtt{F}(\H_1+\H_2,k) \geq \begin{multlined}[t]
             \max \{  \reg\, \mathtt{F}(\H_1,\Fnumber(\H_1)) + \reg\, \mathtt{F}(\H_2,k-\Fnumber(\H_1)), \\
             \ 
             \reg\, \mathtt{F}(\H_1,\Fnumber(\H_1)) + \reg\, \mathtt{F}(\H_2,k+1-\Fnumber(\H_1)) -1  \}.
        \end{multlined}
    \end{align*}
\end{theorem}

\begin{proof}
    Due to \cref{def:sq-free power like}(b) and \cref{lem:del-condition}, the filtration $\{\mathtt{F}(\H_1,i)\}_{i\geq 0}$ and $\{\mathtt{F}(\H_2,j)\}_{j\geq 0}$ are Tor-vanishing. Thus we can apply \cref{thm:reg-Tor-filtration-2}: 
    \begin{align*}
         \reg \, \mathtt{F}(\H_1+\H_2,k) = \begin{multlined}[t]
            \max_{\substack{i\in [a,b]\\j\in [a+1,b]}}  \{ \reg \, \mathtt{F}(\H_1,k-i) + \reg \, \mathtt{F}(\H_2,i) , \\
            \reg \, \mathtt{F}(\H_1,k-j+1) + \reg \, \mathtt{F}(\H_2,j) -1  \}
        \end{multlined}    
    \end{align*}
    for any $k\geq 0$, where 
    \[
    a=\max\{0,k-\Fnumber(\H_1)\} \quad \text{and} \quad b= \min\{k,\Fnumber(\H_2)\}.  
    \]
    Note that $\Fnumber(\H_1+\H_2)= \Fnumber(\H_1)+\Fnumber(\H_2)$. We thus can pick $i=k-\Fnumber(\H_1)\in [a,b]$ and obtain
    \begin{equation}\label{eq:reg-1}
        \reg \, \mathtt{F}(\H_1+\H_2,k) \geq  \reg\, \mathtt{F}(\H_1,\Fnumber(\H_1)) + \reg\, \mathtt{F}(\H_2,k-\Fnumber(\H_1)).
    \end{equation}
    It now suffices to show that
    \begin{equation}\label{eq:reg-2}
        \reg \, \mathtt{F}(\H_1+\H_2,k) \geq \reg\, \mathtt{F}(\H_1,\Fnumber(\H_1)) + \reg\, \mathtt{F}(\H_2,k+1-\Fnumber(\H_1)) -1.
    \end{equation}
    Indeed, if $k=\Fnumber(\H_1+\H_2)= \Fnumber (\H_1)+ \Fnumber(\H_2)$, then $\mathtt{F}(\H_2,k+1-\Fnumber(\H_1))=0$, and (\cref{eq:reg-2}) follows from (\cref{eq:reg-1}). Now we can assume that $k\leq \Fnumber (\H_1)+ \Fnumber(\H_2)-1$. We can now pick $j=k+1-\Fnumber(\H_1)\in [a+1,b]$, and obtain (\cref{eq:reg-2}), as desired.
\end{proof}

We are now ready to prove the main result of this section.

\begin{theorem}\label{thm:lower-bound}
    Let $\H$ be a hypergraph and $k$ an integer where $1\leq k\leq \Fnumber(\H)$. Then for any $k$-admissible $\mathtt{F}$-set $C=\sqcup_{i=1}^r C_i$ of $\H$, we have
    \[
    \reg \left( \mathtt{F}(\H,k) \right) \geq |C|-\sum_{i=1}^r \Fnumber(\H[C_i])+k.
    \]
    Consequently, we have
    \[
    \reg \left( \mathtt{F}(\H,k) \right) \geq \adm^{\mathtt{F}}(\H,k)+k.
    \]
\end{theorem}

\begin{proof}
    We induct on $k$ and $|E(\mathcal{H})|$. If $k=1$, then we have $\Fnumber(\H[C_i])=1$ for any $i\in [r]$ by \cref{def:aim-F} (1) and (3). By \cref{def:aim-F} (4), we have $\reg(\mathtt{F}(\H[C_i]),1) = |C_i|$ for any $i\in [r]$. By \cref{def:aim-F} (2), \Cref{reg sum}, and the Restriction Lemma (\cite[Lemma~4.4]{HHZ04}), we have
    \[
    \reg \left( \mathtt{F}(\H,1) \right)\geq \reg \left( \mathtt{F}(\H[C],1) \right) =  \reg \left( \sum_{i=1}^r \mathtt{F}(\H[C_i],1) \right) = \sum_{i=1}^r|C_i| -r+1,
    \]
    as desired. If $|E(\H)|=1$, then we must have $r=1$ due to \cref{def:aim-F} (1). By \cref{def:aim-F} (3), we have $k\leq \Fnumber(\H[C_1]) \leq k$, and thus $\Fnumber(\H[C_1])=k$. Using \cref{def:aim-F} (4), we have
    \[
    \reg \left( \mathtt{F}(\H,k \right)\geq \reg \left( \mathtt{F}(\H[C_1], k ) \right) =   |C_1| = \left(|C| -\Fnumber(\H[C]) \right)+k,
    \]
    as desired.
    
    By induction, we can assume that $k\geq 2$ and whenever $k'<k$ or $|E(\H')|< |E(\H)|$, we have
    \[
    \reg \left( I(\mathcal{H}')^{\{k'\}} \right) \geq |C'|-\sum_{i=1}^{r'} \Fnumber(\H[C_i'])+k, 
    \]
    where $C'=\sqcup_{i=1}^{r'} C_i'$ is any $k'$-admissible $\mathtt{F}$-set of $\H'$. Now assume that $C=\sqcup_{i=1}^r C_i$ is a $k$-admissible $\mathtt{F}$-set  of $\H$ where the conditions in \cref{def:aim-F} hold. Set $a=|C|-\sum_{i=1}^r \Fnumber(\H[C_i])$, and $a_i=\Fnumber(\H[C_i])$ for each $i\in [r]$. By  the Restriction Lemma (\cite[Lemma~4.4]{HHZ04}), it suffices to show that
    \begin{equation}
        \reg\left( \mathtt{F}(\H[C],k) \right) \geq a+k.
    \end{equation}

    Suppose that $r=1$. Then $k\leq \Fnumber(\H[C])=a_1\leq k$, and thus $\Fnumber(\H[C])=a_1=k$. By \cref{def:aim-F} (4), the ideal $\mathtt{F}(\H[C],k)$ is generated by a monomial of degree $|C|$, and thus we have
    \[
    \reg\left( \mathtt{F}(\H[C],k)  \right) = |C|= \left(|C| - \Fnumber(\H[C])\right)+ \Fnumber(\H[C]) = a+k,
    \]
    as desired. Now we can assume that $r\geq 2$.  Due to condition (2) of \cref{def:aim-F}, the graph $\mathcal{H}[C]$ is exactly the disjoint union of the hypergraphs $\H[C_i]$ where $i$ ranges in $[r]$. Set $\H_1\coloneqq \H[C_1]$, and $\H_2=\sqcup_{i=2}^r \H[C_i]$. Then $\H_1$ and $\H_2$ share no vertex. Moreover, we have 
    \[
    k-\Fnumber(\H_1)=k-\Fnumber(\H[C_1])=k-a_1\geq \sum_{i=2}^r a_i - (r-1)\geq 0
    \]
    and 
    \[k\leq a_1+ \sum_{i=2}^ra_i=\Fnumber(\H_1)+\Fnumber(\H_2),
    \]
    by conditions (2) and (3) of \cref{def:aim-F}.  
    Thus, by \cref{thm:refined-betti-splittings},  we have
    \begin{equation}\label{equ:regularity-betti-F}
        \reg \left( \mathtt{F}(\H[C],k) \right) \geq \max \{ |C_1| + \reg \big( \mathtt{F}(\H_2,k-a_1) \big), \quad |C_1| + \reg \big( \mathtt{F}(\H_2,k-a_1+1) \big) -1  \},
    \end{equation}
    where we used the fact that $\reg \left( \mathtt{F}(\H_1,a_1) \right)$ is the principal ideal generated by a monomial of degree $|C_1|$, from condition (4) of \cref{def:aim-F}. 
    Recall that we have $\Fnumber(\H[C])=\sum_{i=1}^r a_i \leq r+k-1$. If $\sum_{i=1}^r a_i<r+k-1$, then
    \[
    \Fnumber(\H_2)=\sum_{i=2}^r \Fnumber(\H[C_i]) = \sum_{i=2}^r a_i = \sum_{i=1}^r a_i-a_1 < (r+k-1)-a_1 = (r-1)+ (k-a_1),
    \]
    and hence
    \[
     k-a_1\le\Fnumber(\H_2) = \sum_{i=2}^r \Fnumber(\H[C_i]) \leq (r-1) + (k-a_1) - 1.
    \]
    Thus, $\sqcup_{i=2}^r C_i$ is a $(k-a_1)$-admissible $\mathtt{F}$-set of $\H_2$. By induction, we have
    \begin{align*}
        \reg\left( \mathtt{F}(\H_2,k-a_1) \right) \geq (|\sqcup_{i=2}^r C_i| -  \sum_{i=2}^r a_i ) + (k-a_1) &= (a-|C_1| + a_1) + (k-a_1) \\
        &= a -|C_1|+ k.
    \end{align*}
    Thus, by (\ref{equ:regularity-betti-F}), we have
    \[
    \reg \left( \mathtt{F}(\H[C],k) \right) \geq |C_1| + \reg\left( \mathtt{F}(\H_2,k-a_1) \right) \geq  |C_1| + \left(a -|C_1|+ k\right)= a+k, 
    \]
    as desired. Now we can assume that $\sum_{i=1}^r a_i=r+k-1$. By \cref{def:aim-F} (2) and \Cref{lem:Fnumber-properties} (1), we have
    \[
    \Fnumber(H_2) =\sum_{i=2}^r \Fnumber(H[C_i]) = \sum_{i=1}^r a_i -a_1 = (r+k-1)-a_1 = (r-1)+ (k-a_1+1)-1.
    \]
    Since $r\geq 2$, we have
    \[
    k-a_1+1\leq  (r-2) + (k-a_1+1) =  \Fnumber(H_2) = (r-1)+ (k-a_1+1)-1.
    \]
    Thus, $\sqcup_{i=2}^r C_i$ is a $(k-a_1+1)$-admissible $\mathtt{F}$-set of $\H_2$. By induction,    we have
    \begin{align*}
        \reg\left( \mathtt{F}(\H_2,k-a_1+1) \right) &\geq \left(|\sqcup_{i=2}^r C_i| -\sum_{i=2}^{r} \Fnumber(\H[C_i]) \right) + (k-a_1+1)\\
        &= (a-|C_1| + a_1) + (k-a_1+1) \\
        &= a-|C_1|+k+1.
    \end{align*}
    Thus, by (\ref{equ:regularity-betti-F}), we have
    \begin{align*}
        \reg \left(\mathtt{F}(\H[C],k) \right) &\geq |C_1| + \reg\left( \mathtt{F}(\H_2,k-a_1+1) \right) -1\\
        &\geq  |C_1| + \left(a-|C_1|+k+1 \right)-1 \\
        &= a+k, 
    \end{align*}
    as desired. This concludes the proof.
\end{proof}

\begin{remark}\label{ordinary SPLF}
    The (ordinary) squarefree powers $\mathtt F:(\H,k)\mapsto I(\H)^{[k]}$ is an example of a squarefree-power-like function. Indeed, it is easy to see that $I(\H)^{[k]}$ satisfies all the conditions in \Cref{def:sq-free power like} (see also \cite{ChauDasRoySahaSqfOrd}). Moreover, in this situation any $k$-admissible $\mathtt F$-set $C=\sqcup_{i=1}^rC_i$ of $\H$ can be realized as a generalized $k$-admissible matching $M=\sqcup_{i=1}^rM_i$ of $\H$, in the sense of \cite[Definition 3.1]{ChauDasRoySahaSqfOrd}. Thus, by \Cref{thm:lower-bound},
    \[
\reg(I(\H)^{[k]})\ge\max\{|V(M)|-|M|\mid M\text{ is a generalized }k\text{-admissible matching of }\H\}+k.
    \]
This fact was also established in \cite[Theorem 3.11]{ChauDasRoySahaSqfOrd}.
\end{remark}

\begin{remark}\label{rmk: CI reg}
    Given a squarefree-power-like function $\F$. Let $\H$ be the union of pairwise disjoint edges. Then by \cref{prop:coincide-sqf-power}, we may as well assume that $\F$ is the action of taking squarefree powers. Then as shown in \cite[Proposition~3.10]{ChauDasRoySahaSqfOrd}, we have $\reg(\F(\H,k))=\adm^{\F}(\H,k)+k$ for any $1\leq k\leq \Fnumber(\H)$. In other words, our lower bound in \cref{thm:lower-bound} is sharp.
\end{remark}

\begin{remark}
    An analogue of \Cref{thm:refined-betti-splittings} can be obtained for 
$\depth \, \mathtt{F}(\mathcal{H}_1+\mathcal{H}_2,k)$ by applying 
\cite[Theorem~5.3]{THT2020}. On the other hand, deriving a combinatorial 
estimate for $\depth \, \mathtt{F}(\mathcal{H},k)$, in the spirit of 
\Cref{thm:lower-bound} appears to be a considerably more challenging problem. 
In particular, it remains an interesting direction to establish a combinatorial 
lower bound for $\depth \, \mathtt{F}(\mathcal{H},k)$ for all 
$1 \leq k \leq \Fnumber(\mathcal{H})$, where $\mathtt{F}$ denotes a 
squarefree-power-like function.

\end{remark}

\section{Applications to  squarefree symbolic powers}\label{sec: app to sym pwr}

Immediate applications of squarefree-power-like functions include squarefree powers $(\H,k)\mapsto I(\H)^{[k]}$ and squarefree symbolic powers $(\H,k)\mapsto I(\H)^{\{k\}}$. In fact, in the case of squarefree powers, our definition coincides with that of Erey and Hibi \cite{ErHi1}, and $\adm^{\mathtt{F}}$ in this case can be phrased using matchings from graph theory. In a different article \cite{ChauDasRoySahaSqfOrd}, we have explored how sharp this bound is for various squarefree powers of edge ideals. For the rest of the article, we explore this for squarefree symbolic powers and see how our invariants can be related to concepts in combinatorics. First, we provide a proof for why $(\H,k)\mapsto I(\H)^{\{k\}}$ is a squarefree-power-like function. We start with a small lemma, which is a generalization of \cite[Theorem~4.7]{CEM25}. However, as the proof works line-by-line, we refer to that of \cite[Theorem~4.7]{CEM25}.

\begin{lemma}\label{lem:symbolic-restriction}
    Let $\mathcal{H}$ be a hypergraph and $\mathcal{H}_1$ an induced sub-hypergraph of $\H$. Then \[
    \left(I(\H)^{(k)}\right)^{\leq \prod_{x\in V(\H_1)} x} = I(\H_1)^{(k)}. \qedhere
    \] 
\end{lemma}

\begin{proposition}\label{Symbolic SPLF}
    Let $\mathtt{F}(\H,k)=I(\H)^{\{k\}}$ for any hypergraph $\H$ and integer $k$. Then $\mathtt{F}$ is a squarefree-power-like function.
\end{proposition}
\begin{proof}
    The condition (a) of \Cref{def:sq-free power like} follows directly from the definition of squarefree symbolic powers. The condition (b) is due to \cite[Proposition~5.10]{THT2020} and \Cref{lem:sqf-del-condition}. The condition (d) is obtained by taking the squarefree part of the equation in \cite[Theorem 3.4]{THT2020}. As for condition (c), consider a hypergraph $\H$, an induced sub-hypergraph $\H_1$ of $\H$, and any integer $k$. By \cref{lem:symbolic-restriction}, we have
    \[
    \left(I(\H)^{(k)}\right)^{\leq \prod_{x\in V(\H_1)} x} = I(\H_1)^{(k)}.
    \]
    Taking the squarefree part of both sides of the above equality, we obtain
    \[
    \sqf\left(\left(I(\H)^{(k)}\right)^{\leq \prod_{x\in V(\H_1)} x}\right) = I(\H_1)^{\{k\}}.
    \]
 Note that taking the squarefree part is the same as restriction with respect to the product of all variables, and restriction commutes with each other.  Thus,
    \[
    I(\H_1)^{\{k\}}=\sqf\left(I(\H)^{(k)}\right)^{\leq \prod_{x\in V(\H_1)} x}=\left(\sqf(I(\H)^{(k)})\right)^{\leq \prod_{x\in V(\H_1)} x}=\left(I(\H)^{\{k\}}\right)^{\leq \prod_{x\in V(\H_1)} x},
    \]
    as desired.
\end{proof}

Our next goal is to specialize our concepts for squarefree-power-like functions to the case of taking squarefree symbolic powers of the edge ideal of hypergraphs. We will relate them to the combinatorial objects that are known to be linked with symbolic powers namely, independent sets and vertex covers.

Let $\H$ be a hypergraph. Recall that the \textit{independent number} of $\H$, denoted by $\ind(\H)$, is the maximum cardinality of an independent set of $\H$. On the other hand, the \textit{vertex covering number} of $\H$, denoted by $\beta(\H)$, is the minimum cardinality of a vertex cover of $\H$. These two graph-theoretic invariants are instrumental in the study of symbolic powers of edge ideals (see, for instance, \cite{Sullivant2008}). For the squarefree symbolic power, the invariant $\beta(\H)$ plays a pivotal role as seen in \cite{FakhariSymbSq-free} since $I(\H)^{\{k\}}\neq ( 0)$ if and only if $1\le k\le \beta(\H)$. 

In the following definition, we make use of $\alpha(\H)$ to define a certain combinatorial invariant which we call as the $k$-admissible independence number and it turns out that for squarefree symbolic powers of block graphs and certain squarefree symbolic powers of Cohen-Macaulay bipartite graphs the $k$-admissible independence number is same as the regularity of these ideals (see \cref{sec: block} and Section \cref{sec: CM chordal}). 
	
	\begin{definition}\label{def: sym aim for graph}
		Let $\H$ be a hypergraph. Then for any $1\leq k\leq \beta(\H)$, a vertex set $C\subseteq V(\H)$ is called a \emph{$k$-admissible 
        set} of $\H$ if there is a partition $C=\sqcup_{i=1}^rC_i$ such that 
		such that the following holds:
		\begin{enumerate}
			\item for any $i\in [r]$, we have $E(\H[C_i])\neq \emptyset$;
			\item for any edge $e\in E(\H[C])$, we have $e\in E(\H[C_i])$ for some $i\in [r]$;
			\item $k\leq \sum_{i=1}^r \beta(\H[C_i]) \leq r+k-1$;
			\item for each $i\in[r] $, $I(\H[C_i])^{\{\beta(\H[C_i])\}}=\left( \mathbf{x}_{C_i} \right )$.
		\end{enumerate}
		
		For every such $C$, we compute the value of $\ind(\H[C])$. We call the largest value among those, denoted by $\adms(\H,k)$, the \emph{$k$-admissible independence number} of $\H$, i.e.,
		\[
		\adms(\H,k)=\max\{\alpha(\H[C]): C\text{ is a $k$-admissible set of }\H\}
		\]
	\end{definition}
	
	\begin{remark}
		For any $k$-admissible set of $\H$, we have
		\[
		\ind(\H[C])=\sum_{i=1}^r\ind(\H[C_i])=\sum_{i=1}^r (|C_i|- \beta(\H[C_i]))= |C|-\sum_{i=1}^r \beta(\H[C_i]).
		\]
	\end{remark}
	
	\begin{remark}
		For any hypergraph $\H$, the $1$-admissible independence number of $\H$ is the induced matching number of $\H$. In other words, $\adms(\H,1)=\nu(\H)$. To see this, consider $C\subseteq V(\H)$ with a partition $C=\sqcup_{i=1}^r C_i$ such that $\alpha(\H[C])=\adms(\H,1)$. Since $k=1$, from conditions (2) and (3) of  \Cref{def: sym aim for graph}, we have $\beta(\H[C_i])=1$ for all $1\leq i\leq r$. Then, it follows from condition (4) of \Cref{def: sym aim for graph} that $|E(\H[C_i])|=1$. Thus, $\H[C]$ is an induced subgraph of $\H$ consisting of disjoint edges of $\H$, and hence is an induced matching of $\H$. On the other hand, the vertex set of any induced matching of $\H$ can be regarded as a $1$-admissible set of $\H$. 
	\end{remark}
	\begin{lemma}\label{lem: induced adm}
		Let $\H$ be a hypergraph and let $\H'$ be an induced sub-hypergraph of $\H$. Then for any $k\geq 1$, $\adms(\H',k)\leq \adms(\H,k)$.
	\end{lemma}
	\begin{proof}
		Let $C=\sqcup_{i=1}^rC_i\subseteq V(\H')$ be such that $\adms(\H',k)=\ind(\H'[C])$. Since $\H'$ is an induced subgraph of $\H$, it follows that $\H[C]=\H'[C]$, and therefore $\H[C_i]=\H'[C_i]$ for all $i\in [r]$. Hence, $C\subseteq V(\H)$ is a $k$-admissible set of $\H$ and the assertion follows.
	\end{proof}

 We then obtain an application of Theorem~\ref{thm:lower-bound}.

\begin{corollary}\label{cor:lower bound sqf symbolic}
    Let $\H$ be a hypergraph and $1\leq k\leq \height I(\H)$. Let $C=\sqcup_{i=1}^r C_i$ be an  $k$-admissible $\mathtt{F}$-set of $\H$. Then 
    \[
    \reg (I(\H)^{\{k\}} )\geq |C| - \beta(\H[C]) + k = \ind(H[C]) +k.
    \]
    Consequently, we have
    \[
    \reg (I(\H)^{\{k\}}) \geq \adms(\H,k)+k.
    \]
\end{corollary}   
	
	The following lemma provides an effective way to check the condition (4) of \Cref{def: sym aim for graph} for a squarefree monomial ideal.
	 
	\begin{lemma}\label{lem: symb principal}
		Let $\H$ be a hypergraph, and let $k\geq 1$ be an integer such that $I(\H)^{\{k\}}\neq ( 0)$. Then $I(\H)^{\{k\}} =\left ( \x_{V(\H)} \right )$ if and only if for each $x\in V(\H)$, there is a minimal vertex cover $\C$ of $\H$ such that $x\in \C$ and $|\C|=k$. Moreover, if $I(\H)^{\{k\}}=\left ( \x_{V(\H)} \right )$, then $k=\beta(\H)$.
	\end{lemma}
	\begin{proof}
		To prove the `if' part, assume that $I(\H)^{\{k\}} \neq\left \l \x_{V(\H)} \right \r$. Then there is some $x\in V(\H)$ and $f\in I(\H)^{\{k\}}$ such that $x\nmid f$. Since there is a minimal vertex cover $\C$ of $\H$ such that $x\in \C$ and $|\C|=k$, it follows that $f\notin \p_{\C}^{\{k\}}$, as $|\supp(f)\cap C|\leq k-1$. Hence, $f\notin I(\H)^{\{k\}}$, a contradiction. Conversely, assume that $f=\x_{V(\H)}$ is a minimal monomial generator of $I(\H)^{\{k\}}$. Then $f\in \p_{\C}^{\{k\}}$ for every minimal vertex cover $\C$ of $\H$. By way of contradiction, let there be a vertex $x\in V(\H)$ such that every minimal vertex cover $\C$ of $\H$ with $x\in \C$ has $|\C|\geq k+1$. Now, consider $g=\frac{f}{x}$. Then observe that, for any minimal vertex cover $\C$ of $\H$ such that $x\notin \C$, 
		$|\supp(g)\cap \C|\geq k$. Moreover, if $\C'$ is a minimal vertex cover with $x\in \C'$, we also have $|\supp(g)\cap \C'|=|\supp(g)\cap (\C'\setminus \{x\})|\geq k$. This shows that $g\in I(\H)^{\{k\}}$, which is again a contradiction. 
		
		The last statement follows from the `only if' part of the above arguments and the fact that $I(\H)^{\{l\}}\neq\l 0\r$ if and only if $l\le\beta(\H)$.
	\end{proof}

\section{Squarefree Symbolic Powers of Block Graphs}\label{sec: block}

In this section, we show that the regularity of the $k^{\text{th}}$ squarefree symbolic powers of edge ideals of block graphs can be expressed in terms of the $k$-admissible independence number of the graph. The proof of this result relies heavily on the analysis of the $k$-admissible independence number of block graphs. To begin with, we present several applications of \Cref{lem: symb principal} that are necessary for the proof of our main theorem in this section. 

\begin{lemma}\label{lem: attach K_n}
	Let $G$ be a graph and $S\subseteq V(G)$. Let $K_n$ be the complete graph on the vertex set $\{x_1, x_2, \dots , x_n\}$, where $n\ge 2$. Consider the graph $\Gamma$ constructed from $G$ and $K_n$ as follows: 
    \begin{align*}
		V(\Gamma)&=V(G)\sqcup V(K_n),\\
		E(\Gamma)&=E(G)\sqcup E(K_n)\sqcup \{\{u,x_n\}\mid u\in S\}.
	\end{align*}
    Then 
    \begin{enumerate}
        \item[$(1)$] $\beta(\Gamma)=\beta(G)+n-1$.
        \item[$(2)$] If $I(\Gamma)^{\{\beta(\Gamma)\}}=\left\l \x_{V(\Gamma)}\right\r$, then $I(G)^{\{\beta(G)\}}=\left\l \x_{V(G)}\right\r$. For $n\ge 3$, the converse of this statement is also true in the case $n\geq 3$.
    \end{enumerate}
	\end{lemma}
	\begin{proof}
		$(1)$ For $i\in[n-1]$, we set $A_i=\{x_1,\ldots,\widehat{x_i},\ldots,x_n\}$. We claim that, if $\C$ is a minimal vertex cover of $G$, then for any $i\in [n-1]$, $\C\cup A_i$ is a minimal vertex cover of $\Gamma$.
		Indeed, it is easy to see that $\C\cup A_i$ is a vertex cover of $\Gamma$. Now, let $\C'\subseteq \C\cup A_i$ be a minimal vertex cover of $\Gamma$. Observe that $A_i\subseteq \C'$ since $\{x_j,x_i\}\in E(\Gamma)$ for each $x_j\in A_i$. Now, suppose $y\in\C$. Then there exists some $z\in N_G(y)$ such that $z\notin\C$ since $\C$ is a minimal vertex cover of $G$. Consequently, $z\notin\C'$ and since $\{y,z\}\in E(\Gamma)$, we must have $y\in\C'$. Thus $\C'=\C\cup A_i$ is a minimal vertex cover of $\Gamma$, and hence $\beta(\Gamma)\le \beta(G)+n-1$. Next, suppose that $\C''$ is a minimal vertex cover of $\Gamma$ such that $|\C''|=\beta(\Gamma)$. Then by the structure of $\Gamma$ we find that $|\C''\cap\{x_1,\ldots,x_n\}|=n-1$. On the other hand, $\C''\cap V(G)$ is a vertex cover of $G$. Thus, $|\C''\cap V(G)|\ge \beta(G)$. Consequently, $|\C''|=|\C''\cap V(G)|+|\C''\cap \{x_1,\ldots,x_n\}|\ge \beta(G)+(n-1)$. This completes the proof. 
        
        $(2)$ Assume that $I(\Gamma)^{\{\beta(G)+n-1\}}=\left\l\x_{V(\Gamma)}\right\r$. Choose some $y\in V(G)$. Then by \Cref{lem: symb principal}, there exists a minimal vertex cover $\C'$ of $\Gamma$ such that $|\C'|=\beta(G)+n-1$. Proceeding as in the previous paragraph, we see that $\C''$ is a vertex cover of $G$ such that $y\in\C''$ with $|\C''|=\beta(G)$, where $\C''=\C\cap V(G)$. Thus, $\C''$ is a minimal vertex cover of $G$. Therefore, by \Cref{lem: symb principal}, we have $I(G)^{\{\beta(G)\}}=\left\l\x_{V(G)}\right\r$, as desired.

        For the converse part, let us assume that $n\geq 3$ and $I(G)^{\{\beta(G)\}}=\left\l\x_{V(G)}\right\r$. Choose some $y\in V(G)$. Then by \Cref{lem: symb principal}, there is a minimal vertex cover $\C$ of $G$ such that $y\in \C$ and $|\C|=|\beta(G)|$. From part $(1)$, we have that both $\C\cup A_1=\C\cup \{x_2,\ldots,x_n\}$ and $\C\cup A_2=\C\cup\{x_1,x_3,\ldots,x_n\}$ are minimal vertex covers of $\Gamma$ such that $|\C\cup A_1|=|\C\cup A_2|=\beta(G)+n-1$. Note that $y,x_2,\ldots,x_n\in \C\cup A_1$ and $x_1\in \C\cup A_2$. Thus, by \Cref{lem: symb principal}, $I(\Gamma)^{\{\beta(G)+n-1\}}=\left\l\x_{V(\Gamma)}\right\r$.
	\end{proof}

\begin{remark}
    The converse part of the assertion $(2)$ of \Cref{lem: attach K_n} is not necessarily true if we assume $n=2$. For instance, take $G=K_3$ with $S=V(G)=\{y_1,y_2,y_3\}$ and $V(K_2)=\{x_1,x_2\}$. It is easy to see that $\beta(G)=2$ and $I(G)^{\{2\}}=\left\l y_1y_2y_3\right\r$. On the other hand, since $\Gamma$ is nothing but $K_4$ with a whisker attached to one of the vertices, we see that $\beta(\Gamma)=3$ and consequently, $y_1y_2y_3x_2\in I(\Gamma)^{\{3\}}$. Thus, $I(\Gamma)^{\{3\}}\neq \left\l \x_{V(\Gamma)}\right\r$. 
\end{remark}

\begin{lemma}\label{lem: attaching whiskered K_n}
    Let $G$ be a graph and fix a vertex $v\in V(G)$. Let $W(K_n)$ denote the whiskered graph on the complete graph $K_n, n\geq 1$. Let $\Gamma$ be a graph obtained from $G$ and $W(K_n)$, where
    \begin{align*}
		V(\Gamma)&=V(G)\sqcup V(W(K_n)),\\
		E(\Gamma)&=E(G)\sqcup E(W(K_n))\sqcup \{\{v,x\}\mid x\in V(K_n)\}.
	\end{align*}
    Then 
    \begin{enumerate}
        \item[$(1)$] $\beta(\Gamma)=\beta(G)+n$.
        \item[$(2)$] $I(\Gamma)^{\{\beta(\Gamma)\}}=\left\l\x_{V(\Gamma)}\right\r$ if and only if $I(G)^{\{\beta(G)\}}=\left\l\x_{V(G)}\right\r$.
    \end{enumerate}
\end{lemma}
\begin{proof}
    Let $V(K_n)=\{x_1,x_2, \dots , x_n\}$ and $V(W(K_n))=V(K_n)\cup \{y_1,y_2, \dots , y_n\}$ with $E(W(K_n))=E(K_n)\cup \{\{x_i,y_i\}\mid i\in [n]\}$.

    $(1)$ Let $\C$ be a minimal vertex cover of $G$. Then, proceeding as in \Cref{lem: attach K_n}, it is easy to see that $\C\cup V(K_n)$ is a minimal vertex cover of $\Gamma$. Thus, $\beta(\Gamma)\leq \beta(G)+n$. Next, suppose $C$ is a minimal vertex cover of $\Gamma$ such that $|\C|<\beta(G)+n$. Observe that $n\ge |\C\cap V(K_n)|\geq n-1$. If $\C\cap V(K_n)= V(K_n)$, then $\C\setminus V(K_n)$ is a vertex cover of $G$, where $|\C\setminus V(K_n)|< \beta(G)$, a contradiction. If $|\C\cap V(K_n)|= n-1$, then without loss of generality, we can assume that $x_1\notin \C$. In this case, $y_1\in \C$ and hence, $\C'= \C\setminus (V(K_n)\cup \{y_1\})$ is a vertex cover of $G$ such that $|\C'|<\beta(G)$, again a contradiction. Thus we must have $\beta(\Gamma)=\beta(G)+n$, as desired.

    $(2)$ First, let us assume that $I(G)^{\{\beta(G)\}}=\left\l\x_{V(G)} \right \r$. Then by \Cref{lem: symb principal}, for any $u\in V(G)$, there is a vertex cover $\C$ of $G$ such that $u\in \C$ and $|\C|=\beta(G)$. Consequently, $\C'=\C\cup \{x_1, \dots , x_n\}$ is a minimal vertex cover of $\Gamma$ with $u\in \C'$ and $|\C'|=\beta(G)+n=\beta(\Gamma)$. Note that $x_i\in \C'$ for each $ i\in [n]$. Now take any $y_i\in V(\Gamma)$. From our assumption and by \Cref{lem: symb principal}, there is a minimal vertex cover $\C_1$ of $G$ such that $v\in \C_1$ and $|\C_1|=\beta(G)$. Then one can see that $\C_1'=\C_1\cup \{x_1, \dots , x_{i-1}, y_i, x_{i+1}, \dots , x_n\}$ is a vertex cover of $\Gamma$ such that $|\C_1'|=|\C_1|+n=\beta(G)+n=\beta(\Gamma)$. Hence, $\C_1'$ is a minimal vertex cover of $\Gamma$ where $y_i\in \C_1'$ and $|\C_1'|=\beta(\Gamma)$. Thus, using \Cref{lem: symb principal} we have that $I(\Gamma)^{\{\beta(\Gamma)\}}=\left\l\x_{V(\Gamma)} \right \r$. The converse part follows from the repeated use of \Cref{lem: attach K_n}.
\end{proof}

\begin{lemma}\label{aux lemma}
    Let $G$ be a graph and $\{x,y_1\},\{x,y_2\}\in E(G)$ with $\deg_G(y_1)=\deg_G(y_2)=1$. Then $I(G)^{\{\beta(G)\}}\neq \left\l\x_{V(G)}\right\r$.
\end{lemma}
\begin{proof}
    On the contrary, let us assume that $I(G)^{\{\beta(G)\}}= \left\l\x_{V(G)}\right\r$. Then by \Cref{lem: symb principal}, there is a minimal vertex cover $C$ of $G$ such that $y_1\in C$ and $|C|=\beta(G)$. Note that $x\notin C$, and hence $y_2\in C$. Now, $C'=(C\setminus \{y_1,y_2\})\cup \{x\}$ is a vertex cover of $G$ of cardinality $|C'|=\beta(G)-1$, which is a contradiction.
\end{proof}

The following lemma is again an application of \Cref{lem: symb principal}, and the proof of this is similar.

\begin{lemma}\label{lem: complete whisker}
The following results hold for the graph classes constructed from the complete graph.
    \begin{enumerate}
        \item[$(1)$] If $G=K_n$ with $n\geq 2$, then $\beta(G)=n-1$ and  $I(G)^{\{\beta(G)\}}=\left\l\x_{V(G)}\right\r$.
        \item[$(2)$] If $G=W(K_n)$ with $ n\geq 2$, then $I(G)^{\{\beta(G)\}}=\left\l\x_{V(G)}\right\r$.
        \item[$(3)$] Let $n\ge 2$ and $r<n$ be two positive integers. Let $G$ be a graph on $n+r$ vertices with $V(G)=V(K_n)\sqcup \{y_1,\dots , y_r\}$, and $E(G)=E(K_n)\sqcup \{\{u_1,y_1\},\dots , \{u_r, y_r\}\}$, where $u_i\in V(K_n)$ are distinct vertices for $ i\in [r]$. Then $I(G)^{\{\beta(G)\}}\neq \left\l\x_{V(G)}\right\r$.
                
    \end{enumerate}
\end{lemma}

We now recall the notion of a special block in a block graph $G$. Let $\B_G$ denote the set of all blocks in $G$. Following \cite{ChauDasRoySahaSqfOrd}, for $B\in\Bl$ and $u\in V(B)$, we define $\N_G(B,u)\coloneqq\{D\in\Bl\mid V(D)\cap V(B)=\{u\}\}$. Using these notations, a special block is defined in the following way.

\begin{definition}\label{special block}
	Let $G$ be a block graph, and $L$ a block of $G$ with $V(L)=\{u_1,\ldots,u_d\}$. We say that $L$ is a \emph{special block} of $G$ if $L$ is one of the following types.
	\begin{enumerate}[leftmargin=2cm]
		\item[{\bf Type I}] $d\le 2$, and $\N_G(L,u_i)=\emptyset$ for each $i\in[d]$.
		\item[{\bf Type II}] $d\ge 3$, and $\N_G(L,u_i)=\emptyset$ for each $i\in[d-1]$.
		\item[{\bf Type III}] $d\ge 2$, and there exists some $i\in[d-1]$ such that $\N_G(L,u_i)\neq\emptyset$. Moreover, if $\N_G(L,u_i)\neq\emptyset$ for some $i\in[d-1]$, then for each $D\in \N_G(L,u_i)$, $D\cong K_2$.
	\end{enumerate}
\end{definition}

\begin{lemma}\label{special block existence lemma}\cite[Lemma 4.10]{ChauDasRoySahaSqfOrd}
	Let $G$ be a block graph. Then, $G$ contains at least one special block.
\end{lemma}

The inequalities concerning the $k$-independence number described in the next two lemmas are heavily used in the proof of the main theorem (\Cref{thm: block graph main}) of this section. 
	\begin{lemma}\label{lem: ind lemma 1}
		Let $G$ be a block graph and let $L$ be a special block of $G$ with $V(L)=\{u_1,\dots , u_d\}$. Assume that $L$ is a special block of either Type I with $d\ge 2$ or Type III with $d\geq 3$, or Type II. Then for any $1\leq s\leq d$ and $k\geq s$,
		\[
		\adms(G\setminus L, k-s+1)\leq \adms(G,k)-1,    
		\]
		where $G\setminus L$ is the induced subgraph of $G$ on the vertex set $V(G)\setminus V(L)$.
	\end{lemma}
	\begin{proof}
		Let $C = \bigsqcup_{i=1}^r C_i$ be a $(k - s + 1)$-admissible set of $G \setminus L$ such that 
		\[
		\adms(G \setminus L, k - s + 1) = \ind(G[C]).
		\]
		Then by Condition (3) of \Cref{def: sym aim for graph}, $k - s + 1 \leq \sum_{i=1}^{r} \beta(G[C_i]) \leq r + k - s.$ We divide the proof in two cases.

		\noindent
		\textbf{Case I:} Assume that $L$ is a special block of Type I with $d= 2$. In this case, $s\in \{1,2\}$. We sketch the proof for the $s=1$ case only, as the case $s=2$ can be proved by similar arguments. Take $C_{r+1}=\{u_1,u_{2}\}$ and set $C'=\bigsqcup_{i=1}^{r+1} C_i$. Then from the structure of $G$ we have $E(G[C_{r+1}])\neq\emptyset$, and if $e\in E(G[C'])$, then $e\in E(G[C_i])$ for some $i\in [r+1]$. Moreover, $G[C_{r+1}]$ is a connected component of $G$ with $\beta(G[C_{r+1}])=1$ and hence,
		\[
		I(G[C_{r+1}])^{\{\beta(G[C_{r+1}])\}} =\l \x_{C_{r+1}}\r.
		\] 
		Observe that 
\[		\beta(G[C']) = \sum_{i=1}^{r+1} \beta(G[C_i]) = \beta(G[C]) + 1,
\]
		and thus, 
        \[
        k+1 \leq \sum_{i=1}^{r+1} \beta(G[C_i]) \leq r + k.
        \]
        From this, it follows that
        \[
        k \leq \sum_{i=1}^{r+1} \beta(G[C_i]) \leq (r+1) + k - 1.
        \]
        Therefore, $C'$ is a $k$-admissible set of $G$. Note that
		\[
		\ind(G[C']) = |C'| - \beta(G[C']) = |C| + 2 - \beta(G[C]) + 1 = \ind(G[C]) + 1.
		\]
		This proves the inequality $\adms(G\setminus L,k) \leq \adms(G,k) - 1$. 
		
		\noindent
		\textbf{Case II:} Assume that $L$ is a special block of Type III with $d\geq 3$ or Type II. Then we have the following two subcases. 
		
		\noindent
		\textbf{Case IIA:} Let us assume first that $s<d$. In this case we take $C_{r+1} \subseteq \{u_1, \dots, u_{d-1}\}$ such that $|C_{r+1}| = s$ and set $C'=\bigsqcup_{i=1}^{r+1} C_i$. Again, by the structure of $G$ we have $E(G[C_{r+1}])\neq\emptyset$, and if $e\in E(G[C'])$, then $e\in E(G[C_i])$ for some $i\in [r+1]$. Note that $G[C_{r+1}]$ is a complete graph with $\beta(G[C_{r+1}]) = s - 1,$ and by \Cref{lem: complete whisker}, $I(G[C_{r+1}])^{\{s - 1\}} = \l \x_{C_{r+1}}\r$. Further, observe that 
        \[
        k \leq \sum_{i=1}^{r+1} \beta(G[C_i]) \leq r + k - 1 < (r + 1) + k - 1.
        \]
		Thus, $C' = \bigsqcup_{i=1}^{r+1} C_i$ is a $k$-admissible set of $G$. Now,
		\begin{align*}
			\ind(G[C'])=|C'|-\beta(G[C'])=|C|-\beta(G[C])+1=\ind(G[C])+1=\adms(G\setminus W,k-s+1)+1
		\end{align*}
		This shows that whenever $s\neq d$, we have $\adms(G\setminus W,k-s+1)\leq \adms(G,k)-1$.
		
		\noindent
		\textbf{Case IIB:} Consider the remaining part, i.e., $s = d$. Then we have
		\[
		k - d + 1 \leq \sum_{i=1}^{r} \beta(G[C_i]) \leq r + k - d.
		\]
		Now, if $\sum_{i=1}^{r} \beta(G[C_i]) \geq k - d + 2$, then we can take $C_{r+1} = \{u_1, \dots, u_{d-1}\}$ and set $C'=\bigsqcup_{i=1}^{r+1} C_i$. As before, by the structure of $G$ we have $E(G[C_{r+1}])\neq\emptyset$, and if $e\in E(G[C'])$, then $e\in E(G[C_i])$ for some $i\in [r+1]$. Moreover, since $G[C_{r+1}]$ is a complete graph with $\beta(G[C_{r+1}]) = d - 2$, using \Cref{lem: complete whisker} we have $I(G[C_{r+1}])^{\{d-2\}} = \l \x_{C_{r+1}}\r$. Furthermore, we have the inequalities 
        \[
        k \leq \sum_{i=1}^{r+1} \beta(G[C_i]) \leq r + k - 2 < (r + 1) + k - 1.
        \]
		Thus, $C' = \bigsqcup_{i=1}^{r+1} C_i$ is a $k$-admissible set of $G$, and consequently, $\ind(G[C'])=\ind(G[C]) + 1$. This, in turn, implies that $\adms(G \setminus W, k - d + 1) \leq \adms(G, k) - 1$.
		
		Finally, assume that $\sum_{i=1}^{r} \beta(G[C_i]) = k - d + 1$. In this case, we take $C'=\widetilde{C}$ as a single partition, where $\widetilde{C}=C\cup\{u_1,\ldots,u_d\}$. Thus, Conditions (1) and (2) of \Cref{def: sym aim for graph} are automatically satisfied for $C'$. Observe that $I(G[C])^{\{\beta(G[C])\}}=\l \x_{C}\r$ because of Conditions (2) and (4) in \Cref{def: sym aim for graph}.  Now, since $d\ge 3$, using \Cref{lem: attach K_n}, we obtain $\beta(G[C']) = \beta(G[C]) + d - 1$ and $I(G[C'])^{\{\beta(G[C'])\}} =\l \x_{C'}\r$. Thus, in order to show that $C'$ is a $k$-admissible set, it remains to check Condition (4) in \Cref{def: sym aim for graph}. However, since $C'$ is a single partition and $\beta(G[C']) = \beta(G[C]) + d - 1=k$, we have what we desired. Thus, $C'$ is a $k$-admissible set and 
		\[
		\ind(G[C']) 
		= |C'| - \beta(G[C']) = |C| + d - (\beta(G[C]) + d - 1)= \ind(G[C]) + 1.
		\]
		Consequently, $\adms(G\setminus W,k-d+1)\leq \adms(G,k)-1$, in this case too. This completes the proof of the lemma.
        \end{proof}

	\begin{lemma}\label{lem: ind lemma 2}
		Let $G$ be a block graph and $L=\{u_1,\dots , u_d\}$ is a special block of Type III such that for each $i\in [d-1]$, $N_G[u_i]=V(L)\cup \{v_{i,1},\dots , v_{i,t_i}\}$, where $t_i$ is a positive integer. Then for all $i\in [d-1]$ and $k\geq 2$,
		\[
		\adms(G\setminus u_i,k-1)=\adms(G\setminus \{u_i,v_{i,1},\dots , v_{i,t_i}\},k-1)\leq \adms(G,k)-1.
		\]
	\end{lemma}
	\begin{proof}
		Without loss of generality, we may assume that $i=1$. To simplify notations, we set $ G' = G \setminus \{u_1, v_{1,1}, \dots, v_{1,t_1} \}$. It is easy to see that $\adms(G\setminus u_1,k-1)= \adms(G',k-1)$. Thus it remains to show that $\adms(G',k-1)\le \adms(G,k)-1$. Let \( C = \bigsqcup_{i=1}^r C_i \) be a \((k-1)\)-admissible set of \( G' \) such that $\adms(G', k-1) = \ind(G[C])$. Then
		\begin{align}\label{ineq321}
		k - 1 \leq \sum_{i=1}^{r} \beta(G[C_i]) \leq r + (k - 1) - 1.
		\end{align}
		If \( N_G(u_1) \cap C = \emptyset \), then take \( C_{r+1} = \{u_1, v_{1,1}\} \) and set \( C' = \bigsqcup_{i=1}^{r+1} C_i \). Then it follows from the structure of $G$ that $E(G[C_{r+1}])\neq\emptyset$, and if $e\in E(G[C'])$, then $e\in E(G[C_i])$ for some $i\in [r+1]$. Observe that $\beta(G[C_{r+1}]) =1$. Therefore,	
		\[k \leq \sum_{i=1}^{r+1} \beta(G[C_i]) \leq r + k - 1 < (r+1) + k - 1.
		\]
		Thus, \( C' = \bigsqcup_{i=1}^{r+1} C_i \) is a \( k \)-admissible set of \( G \). Note that $\ind(G[C']) = \ind(G[C]) + 1$,
		and therefore, $\adms(G', k - 1)\le \adms(G, k) - 1$, in this case.

		Next, we assume \( N_G(u_1) \cap C \neq \emptyset \). Since $G[\{u_1,\ldots,u_d\}]$ is a complete graph, using Condition (2) of \Cref{def: sym aim for graph}, we can say that there exists some \( \ell \in [r] \) such that
		\[
		N_G(u_1) \cap C_\ell \neq \emptyset \quad \text{and} \quad N_G(u_1) \cap C_j = \emptyset \text{ for all } j \in [r], j \neq \ell.
		\]
		We first fix some notations for the rest of the proof. Let $T=\{u_2,u_3,\dots , u_d\}$. For $ j\in [d-1]$, we set $S_j:= N_G(u_j)\setminus V(L)= \{v_{j,1}, \dots , v_{j,{t_j}}\}$. Note that for each $j\in [d-1]$, by our choice of $L$, $v_{j, q}, 1\leq q\leq t_j$ are leaf vertices for each $q\in [t_j]$. Further, we set $S=\cup_{j=1}^{d-1}S_j$, and $\Lambda =S\cup T$. Then consider the following two cases:
		
		\noindent
		\textbf{Case A.} First we assume that $C_{\ell}\cap S=\emptyset$. Then $\Lambda \cap C_\ell\subseteq T$. We have the following two subcases to consider.
		
		\noindent
		\textbf{Subcase A.1.} $u_d\notin C_{\ell}$. Then using the given hypothesis $C_{\ell}\cap S=\emptyset$ and $ I(G[C_{\ell}])^{\beta(G[C_{\ell}])}=\left\l \x_{C_{\ell}} \right \r$, we can write
		$N_G(u_1) \cap C_\ell=\{u_{i_1},\dots , u_{i_s}\}$, where $s\ge 2$ and $d\notin\{i_1,,\ldots,i_s\}$. Now, take $C_{\ell}'=C_{\ell}\cup \{v_{i_1,1}, v_{i_2, 1},\dots , v_{i_{s},1}\}$ and set
		 \[C''=(\sqcup_{i\in[r], i\neq \ell}C_i)\sqcup C_{\ell}'.\]
		  Observe that we can write $C_{\ell}=A_1\sqcup A_2$, where $A_1=N_G(u_1)\cap C_{\ell}$ and $A_2\subseteq V(G)\setminus (\cup_{i=1}^{d-1}N_G[u_i])$. Moreover, $\beta(G[C_{\ell}])=\beta(G[A_1])+\beta(G[A_2])$, $I(G[A_2])^{\{\beta(A_2)\}}=\l\x_{A_2}\r$ whenever $A_2\neq\emptyset$, and $I(G[A_1])^{\{\beta(A_1)\}}=\l\x_{A_1}\r$. Let us take $A_1'=A_1\cup\{v_{i_1,1}, v_{i_2, 1},\dots , v_{i_{s},1}\}$. Then $G[A_1']$ is a whisker graph, and thus by \Cref{lem: complete whisker}, $\beta(G[A_1'])=s$ and $I(G[A_1'])^{\{\beta(A_1')\}}=\l\x_{A_1'}\r$. Consequently, $ I(G[C_{\ell}'])^{\beta(G[C_{\ell}'])}=\left\l \x_{C_{\ell}'} \right \r$. Also, it follows that $e\in E(G[C''])$ implies either $e\in E(G[C_i])$ for some $i\in [r]\setminus\{\ell\}$ or $e\in E(G[C_{\ell}'])$. Moreover, since $\beta(G[A_1'])=s=\beta(G[A_1])+1$, from \Cref{ineq321} we have
		\[
		k \leq \left( \sum_{i \in [r], i \neq \ell} \beta(G[C_i]) \right) + \beta(G[C_\ell']) \leq r + k - 1.
		\]
		Thus, \( C'' \) is a \( k \)-admissible set of \( G \) such that $	\ind(G[C'']) = \ind(G[C]) + s-1$. In particular, $\adms(G', k - 1)\le \adms(G, k) - 1$, in this case too.
		
		\noindent
		\textbf{Subcase A.2.} $u_d\in C_{\ell}$.
		
		\noindent
		\textbf{Subcase A.2.1.} $N_G(u_1) \cap C_\ell=\{u_d\}$. In this case we take $C_{\ell}'=C_{\ell}\cup \{u_1, v_{1,1}\}$, and it follows from \Cref{lem: attaching whiskered K_n} that $ I(G[C_{\ell}'])^{\beta(G[C_{\ell}'])}=\l \x_{C_{\ell}'}  \r$, where	$\beta(G[C_{\ell}'])=\beta(G[C_{\ell}])+1$.
		Now, if we set \[
		C''=(\sqcup_{i\in[r], i\neq \ell}C_i)\sqcup C_{\ell}',
		\]
		 then by the construction we see that $e\in E(G[C''])$ implies either $e\in E(G[C_i])$ for some $i\in [r]\setminus\{\ell\}$ or $e\in E(G[C_{\ell}'])$. Furthermore, from \Cref{ineq321} we have 
		 \[
		 k \leq \left( \sum_{i \in [r], i \neq \ell} \beta(G[C_i]) \right) + \beta(G[C_\ell']) \leq r + k - 1.
		 \]
		  Therefore, \( C'' \) is a \( k \)-admissible set of \( G \) such that $\ind(G[C'']) = \ind(G[C]) + 1$. Hence, $\adms(G', k - 1)\le \adms(G, k) - 1$, in this case also. 
		
		\noindent
		\textbf{Subcase A.2.2.} $N_G(u_1) \cap C_\ell=\{u_{i_1},\dots , u_{i_s}, u_d\}$ for some $s\ge 1$. In this case, we proceed to show that $s$ must be $1$. Indeed, if $s\ge 2$, then take $C_{\ell}'=(C_{\ell}\setminus \{u_d\}) \cup \{v_{i_1,1}, v_{i_2, 1},\dots , v_{i_{s},1}\}$ and set \[
		C''=(\sqcup_{i\in[r], i\neq \ell}C_i)\sqcup C_{\ell}'.
		\]
		 Since $ I(G[C_{\ell}])^{\beta(G[C_{\ell}])}=\l \x_{C_{\ell}}\r$, using \Cref{lem: attach K_n} and \Cref{lem: complete whisker} we have $ I(G[C_{\ell}'])^{\beta(G[C_{\ell}'])}=\l \x_{C_{\ell}'}\r$, where $\beta(G[C_{\ell}'])=\beta(G[C_{\ell}])$. It is easy to see that $e\in E(G[C''])$ implies either $e\in E(G[C_i])$ for some $i\in [r]\setminus\{\ell\}$ or $e\in E(G[C_{\ell}'])$. Thus,  \( C'' \) is a \( (k-1) \)-admissible set of \( G'' \) such that \[
		 \ind(G[C'']) = \ind(G[C]) + s-1> \ind(G[C]),
		 \] 
		 a contradiction to the fact that $\adms(G', k-1) = \ind(G[C])$. Thus, we have $s=1$. In this case we take $C_{\ell}'=(C_{\ell}\setminus \{u_d\}) \cup \{v_{i_1, 1}, u_1, v_{1,1}\}$, and set
		  \[
		 C''=(\sqcup_{i\in[r], i\neq \ell}C_i)\sqcup C_{\ell}'.
		 \]
		 Then, proceeding as before we see that $C''$ is a $k$-admissible set of $G$ such that $\ind(G[C'']) = \ind(G[C]) + 1$, and consequently, $\adms(G', k - 1)\le \adms(G, k) - 1$, in this case.

		\noindent
		\textbf{Subcase B.} Assume that $C_{\ell}\cap S\neq \emptyset$. In this case, by \Cref{aux lemma} we find that $|C_{\ell}\cap S_j|\le 1$ for each $j\in[d-1]$. Based on this observation, we have the following two subcases to consider. 
		
		\noindent
		\textbf{Subcase B.1.} $u_d\notin C_{\ell}$. Then by \Cref{lem: complete whisker}(3) and the fact that $C_{\ell}\cap S\neq \emptyset$, we find that $G[C_{\ell}\cap \Gamma]$ is a complete whisker graph. Thus if we take $C_{\ell}'=C_{\ell} \cup \{u_1, v_{1,1}\}$, and set
		\[
		C''=(\sqcup_{i\in[r], i\neq \ell}C_i)\sqcup C_{\ell}',
		\]
		then from \Cref{lem: complete whisker}(2) we have $	\beta(G[C_{\ell}'])=\beta(G[C_{\ell}])+1$ and $I(G[C_{\ell}'])^{\beta(G[C_{\ell}'])}= \l \x_{C_{\ell}'} \r$.
Hence, 
		\[
		k \leq \left( \sum_{i \in [r], i \neq \ell} \beta(G[C_i]) \right) + \beta(G[C_\ell']) \leq r + k - 1.
		\]
		Also, observe that $e\in E(G[C''])$ implies either $e\in E(G[C_i])$ for some $i\in [r]\setminus\{\ell\}$ or $e\in E(G[C_{\ell}'])$. Thus, \( C'' \) is a \( k \)-admissible set of \( G \) such that
		$\ind(G[C'']) = \ind(G[C]) + 1$, and consequently, $\adms(G', k - 1)\le \adms(G, k) - 1$, in this case also.
		
		\noindent
		\textbf{Subcase B.2.} $u_d\in C_{\ell}$. In this case, without loss of generality, we can assume that $\Lambda \cap C_{\ell}= \{ u_2, \dots , u_s, v_{2,1}, \dots , v_{t,1}, u_d\}$, where $2\leq t\leq s\le d-1$.
		
		\noindent
		\textbf{Subcase B.2.1.} $t=s$. Take $C_{\ell}'=C_{\ell} \setminus  \{ u_2, \dots , u_s, v_{2,1}, \dots , v_{s,1}\}$. By repeated applications of \Cref{lem: attach K_n}(2) with $n=2$, we get $\beta(G[C_{\ell}'])=\beta(G[C_{\ell}])-(s-1)$ and $I(G[C_{\ell}'])^{\{\beta(G[C_{\ell}'])\}}= \l \x_{C_{\ell}'} \r$.
		Moreover \( \ind(G[C_{\ell}']) = \ind(G[C]) -(s-1).\) Now, we set $C_{\ell}''=C_{\ell}'\cup \{ u_1, u_2, \dots , u_s, v_{1,1}, v_{2,1}, \dots , v_{s,1}\}$. Then by \Cref{lem: attaching whiskered K_n}, we have $\beta(G[C_{\ell}''])=\beta(G[C_{\ell}])+1$ and $I(G[C_{\ell}''])^{\{\beta(G[C_{\ell}''])\}}=\l \x_{C_{\ell}''} \r$.
		Moreover, \( \ind(G[C_{\ell}'']) = \ind(G[C]) +1\). Taking \(C''=(\sqcup_{i\in[r], i\neq \ell}C_i)\sqcup C_{\ell}''\), we see that $e\in E(G[C''])$ implies either $e\in E(G[C_i])$ for some $i\in [r]\setminus\{\ell\}$ or $e\in E(G[C_{\ell}''])$. Furthermore, from \Cref{ineq321} we get
		\[
		k \leq \left( \sum_{i \in [r], i \neq \ell} \beta(G[C_i]) \right) + \beta(G[C_\ell'']) \leq r + k - 1.
		\]
		Thus, \( C'' \) is a \( k \)-admissible set of \( G \) such that $\ind(G[C'']) = \ind(G[C]) + 1$, and hence $\adms(G', k - 1)\le \adms(G, k) - 1$, in this case.
		
		\noindent
		\textbf{Subcase B.2.2.} $t<s$. Take $C_{\ell}^{1}=C_{\ell} \setminus  \{ u_2, \dots , u_t, v_{2,1}, \dots , v_{t,1}\}$. By repeated applications of \Cref{lem: attach K_n}(2) with $n=2$, we get $\beta(G[C_{\ell}^{1}])=\beta(G[C_{\ell}])-(t-1)$ and $I(G[C_{\ell}^{(1)}])^{\{\beta(G[C_{\ell}^{1}])\}}=\l \x_{C_{\ell}^{1}} \r$.
		Moreover \( \ind(G[C_{\ell}^{(1)}]) = \ind(G[C_{\ell}]) -(t-1)\). Next, we take  $C_{\ell}^{2}=C_{\ell}^{1} \setminus  \{ u_{t+1}, \dots , u_s,u_d\}$. Then by \Cref{lem: attach K_n}(2) it follows that 
		\[
		\beta(G[C_{\ell}^{2}])=\beta(G[C_{\ell}^{1}])-(s-t)=\beta(G[C_{\ell}])-s+1,\]
		and $I(G[C_{\ell}^{2}])^{\beta(G[C_{\ell}^{2}])}=\l \x_{C_{\ell}^{2}}\r$. Moreover, \[
		\ind(G[C_{\ell}^{2}]) = \ind(G[C_{\ell}^{1}]) -1=\ind(G[C_{\ell}]) -t.
		\]
		Finally, we take $C_{\ell}^{3}=C_{\ell}^{2}\sqcup \{ u_1, \dots , u_s, v_{1,1}, \dots , v_{s,1}\}$. Then it follows from \Cref{lem: attaching whiskered K_n} that 
		\[
		\beta(G[C_{\ell}^{3}])=\beta(G[C_{\ell}^{2}])+s=\beta(G[C_{\ell}])+1,
		\]
		and $I(G[C_{\ell}^{3}])^{\beta(G[C_{\ell}^{3}])}=\l \x_{C_{\ell}^{3}} \r$. Moreover, 
		\[ 
		\ind(G[C_{\ell}^{3}]) = \ind(G[C_{\ell}^{2}])+s=\ind(G[C_{\ell}])+ (s-t).
		\]
		Next, we take \(C''=(\sqcup_{i\in[r], i\neq \ell}C_i)\sqcup C_{\ell}^{3}\), then it follows that 
		\[
		k \leq \left( \sum_{i \in [r], i \neq \ell} \beta(G[C_i]) \right) + \beta(G[C_\ell^{3}]) \leq r + k - 1.
		\]
		observe from the structure of $G$ that  if $e\in E(G[C''])$, then either $e\in E(G[C_i])$ for some $i\in [r]\setminus\{\ell\}$ or $e\in E(G[C_{\ell}^3])$. Thus, \( C'' \) is a \( k \)-admissible set of \( G \) such that
		\[
		\ind(G[C'']) = \ind(G[C]) + (s-t)\geq \ind(G[C])+1.
		\]
Hence, $\adms(G', k - 1)\le \adms(G, k) - 1$, in this case too. This completes the proof of the lemma.	
\end{proof}

	We are now in a position to prove the main theorem of this section.
	
	\begin{theorem}\label{thm: block graph main}
		Let $G$ be a block graph on $n$ vertices and let $L$ be a special block with $V(L)=\{u_1,\dots , u_d\}$. Then, for each $1\le k\le \nu(G)$, we have the following:
		\begin{enumerate}
			\item $\reg(I(G)^{\{k\}})\le \adms(G,k)+k$.
			\item If $\N_G(L,u_i)= \emptyset$ for some $i\in [d-1]$, then $\reg(I(G)^{\{k\}}:u_i)\le \adms(G,k)+k-1$.
			\item If $\N_G(L,u_i)\neq \emptyset$ for all $i\in [d-1]$, then 
			\[
			\reg(I(G)^{\{k\}}:u_i)\le \adms(G,k)+k-1,
			\]
			for all $i\in [d-1]$.
		\end{enumerate}
		Consequently, 
		\[
		\reg(I(G)^{\{k\}})=\adms(G,k)+k.
		\]
	\end{theorem}
	\begin{proof}
		We first prove (1), (2), and (3) simultaneously by induction on $n$. If $n \leq 2$, then $k = 1$ and thus by \cite[Theorem 6.8]{HaVanTuyl2008} 
		\[
		\reg(I(G)^{\{1\}}) = \nu(G) + 1 = \adms(G,1) + 1.
		\]
		Moreover, (2) is trivially satisfied, and (3) is vacuously true. Therefore, we may assume that $n \geq 3$. For such an $n$, we first consider the case $k = 1$. Then (1) follows directly from \cite[Theorem 6.8]{HaVanTuyl2008}. For (2) and (3), note that $\reg((I(G)^{\{1\}}:u_i))=\reg(I(G\setminus V(L)))$. Thus the inequality $\reg(I(G)^{\{1\}}:u_i)\le \adms(G,1)$ can be verified using \Cref{lem: ind lemma 1}, \Cref{lem: ind lemma 2}, and the induction hypothesis. So, we further assume that $k \geq 2$. We now divide the proof into two cases.

		\noindent
		\textbf{Case A.} Let $L$ be a special block of Type I in $G$. In this case, either $L$ is an isolated vertex or $L \cong K_2$. If $L$ is an isolated vertex of $G$, then by the induction hypothesis and by \Cref{lem: induced adm}, we have
		\begin{align*}
			\reg(I(G)^{\{k\}}) &= \reg(I(G \setminus V(L))^{\{k\}}) \\
			&\le \adms(G \setminus V(L),k) + k \\
			&\le \adms(G,k) + k.
		\end{align*}
		Thus, statement (1) is satisfied, and statements (2) and (3) are vacuously true.
		
		Now, suppose $L \cong K_2$ and $V(L) = \{u_1, u_2\}$. Then, $\N_G(B,u_i)= \emptyset$ for $i=1,2$. Thus, statement (3) is vacuously true. For statement (2), we set $W=\{u_1,u_2\}$. Then by \Cref{lem: symb sq-free 2}, we get
		\begin{align*}
			\reg(I(G)^{\{k\}}:u_1)& \leq \max\{\reg(I(G\setminus u_{2})^{\{k\}}:u_1),\; \reg(I(G)^{\{k\}}:u_1u_2)+1 \}\\
			&= \max\{\reg(I(G\setminus \{u_1,u_2\})^{\{k\}}),\; \reg(I(G\setminus \{u_1,u_2\})^{\{k-1\}})+1 \} \\
			&\leq  \max\{\adms(G\setminus \{u_1,u_2\}, k)+k,\; \adms(G\setminus \{u_1,u_2\}, k-1)+k \},
		\end{align*}
		where the second equality follows from \Cref{lem: symb sq-free 1}, and the third inequality follows from the induction hypothesis of the statement (1). Finally, the assertion $\reg(I(G)^{\{k\}}:u_1) \leq \adms(G,k) +k - 1$ follows from \Cref{lem: ind lemma 1}. This proves statement (2). For statement (1), observe that 
		\begin{align*}
			\reg(I(G)^{\{k\}}+\l u_1\r)&=\reg(I(G\setminus \{u_1,u_2\})^{\{k\}}+\l u_1\r)\\
			&\le\reg(I(G\setminus \{u_1,u_2\})^{\{k\}})\\
			&\le \adms(G\setminus \{u_1,u_2\}, k)+k\\
			&\le \adms(G, k)+k,
		\end{align*}
		where the inequality in the second line follows from \Cref{regularity lemma}(i), the inequality in the third line follows from the induction hypothesis on statement (1), and the last inequality follows from \Cref{lem: induced adm}. We can now use \Cref{regularity lemma}(iii) and statement (2) to conclude that $\reg(I(G)^{\{k\}})\le \adms(G,k)+k$, which proves (1).
		
		\noindent
		\textbf{Case B.} Let us assume that $L$ is a special block of $G$ which is either of Type II or of Type III. Let $V(L)=\{u_1,\dots , u_d\}$, where $d\geq 2$. 
		
		\noindent
		\textbf{Proof of (2):} Let us assume that there is a vertex $u_i\in V(L)$ with $i\in[d-1]$ such that $\N_G(L, u_i)=\emptyset$. Then observe that we must have $d\geq 3$. Without any loss of generality, assume that $i=1$. We set $W=\{u_1,\dots ,u_d\}$. Then by \Cref{lem: symb sq-free 2}, we get
		\begin{equation}\label{eqn: simplicial colon}
			\begin{split}
				\reg(I(G)^{\{k\}}:u_1)&\leq \max \{\reg (I(G\setminus U_1)^{\{k\}}: \mathbf{x}_{U_2})+|U_2|-1\mid u_1\in U_2, U_1\cap U_2=\emptyset, U_1\cup U_2=W\}\\
				&\leq \max \{\reg (I(G\setminus W)^{\{k-|U_2|+1\}})+|U_2|-1\mid u_1\in U_2, U_2\subseteq W\}\\
				&\leq \max \{\adms(G\setminus W, k-|U_2|+1) +k \mid  u_1\in U_2, U_2\subseteq W\} 
			\end{split}
		\end{equation}
		where the inequality in the second line follows from \Cref{lem: symb sq-free 1}, and the inequality in the third line follows from the induction hypothesis of the statement (1). Finally, it follows from \Cref{lem: ind lemma 1} that $\adms(G\setminus W, k-|U_2|+1)\leq \adms(G,k)-1$, for any $U_2\subseteq W$ with $1\leq |U_2|\leq d$. Hence, \Cref{eqn: simplicial colon}, implies that $\reg(I(G)^{\{k\}}:u_1)\leq \adms(G,k)-1$. This completes the proof of (2).

		\noindent
		\textbf{Proof of (3):} Assume that for every $u_i\in V(L), 1\leq i\leq d-1$, $\N_G(L, u_i)\neq \emptyset$. Then, $L$ is a special block of Type III. First, we consider the case when $d=2$. Then $V(L)=\{u_1,u_2\}$, and assume that $N_G(u_1)=\{u_2, v_{1,1},\dots , v_{1,s}\}, s\geq 1$. We set $W=\{u_1, v_{1,1}\}$.  Then by \Cref{lem: symb sq-free 2}, we get
		\begin{align*}
			\reg(I(G)^{\{k\}}:u_1)& \leq \max\{\reg(I(G\setminus v_{1,1})^{\{k\}}:u_1),\; \reg(I(G)^{\{k\}}:u_1v_{1,1})+1 \}\\
			&= \max\{\reg(I(G\setminus v_{1,1})^{\{k\}}:u_1),\; \reg(I(G\setminus \{u_1, v_{1,1}\})^{\{k-1\}})+1 \} \\
			&= \max\{\reg(I(G\setminus v_{1,1})^{\{k\}}:u_1),\; \reg(I(G\setminus \{u_1,v_{1,1},\dots , v_{1,s}\})^{\{k-1\}})+1 \} \;\;\; \\
			&\leq \max\{\reg(I(G\setminus v_{1,1})^{\{k\}}:u_1),\; \adms(G\setminus \{u_1,v_{1,1},\dots , v_{1,s}\},k-1)+k \}
		\end{align*}
		where the equality in the second line follows from \Cref{lem: symb sq-free 1}, and the inequality in the fourth line follows from the induction hypothesis of the statement (1).
		
		\noindent
		Now, by \Cref{lem: ind lemma 2}, it follows that  $\adms(G\setminus \{u_1,v_{1,1},\dots , v_{1,s}\},k-1)\leq \adms(G,k)-1$. Thus, it suffices to show that $\reg(I(G\setminus v_{1,1})^{\{k\}}:u_1)\leq \adms(G,k)+k-1$. For this, we again consider $W=\{u_1,v_{1,2}\}\subseteq V(G\setminus v_{1,1})$, and proceed similarly. After $s$-many steps, we finally obtain 
		\begin{align*}
			\reg(I(G)^{\{k\}}:u_1) &\leq \max\{\reg(I(G\setminus \{v_{1,1}, \dots , v_{1,s}\})^{\{k\}}:u_1), \\ & \hspace{3cm} \adms(G\setminus \{u_1,v_{1,1},\dots , v_{1,s}\},k-1)+k\}.
		\end{align*}
		Thus, it is enough to show that $\reg(I(G\setminus \{v_{1,1}, \dots , v_{1,s}\})^{\{k\}}:u_1)\leq \adms(G,k)+k-1$. Now, we set $G'=G\setminus \{v_{1,1}, \dots , v_{1,s}\}$ and $W=\{u_1,u_2\}\subseteq V(G')$. Then by applying \Cref{lem: symb sq-free 2} again, we get
		\begin{align*}
			\reg(I(G')^{\{k\}}:u_1)& \leq \max\{\reg(I(G'\setminus u_{2})^{\{k\}}:u_1),\; \reg(I(G')^{\{k\}}:u_1u_2)+1 \}\\
			&= \max\{\reg(I(G'\setminus \{u_1,u_2\})^{\{k\}}),\; \reg(I(G'\setminus \{u_1,u_2\})^{\{k-1\}})+1 \} \\
			&\leq  \max\{\adms(G'\setminus \{u_1,u_2\}, k)+k,\; \adms(G'\setminus \{u_1,u_2\}, k-1)+k \}\\
			&=  \max\{\adms(G\setminus \{u_1,u_2, v_{1,1}, \dots , v_{1,s}\}, k)+k,\; \\ & \hspace{5cm} \adms(G\setminus \{u_1,u_2, v_{1,1}, \dots , v_{1,s}\}, k-1)+k \}.     
		\end{align*}
		where the second equality follows from \Cref{lem: symb sq-free 1}, and the third inequality follows from the induction hypothesis of the statement (1).
		
		\noindent
		Let us set $G''=G\setminus \{u_1,u_2, v_{1,1}, \dots , v_{1,s}\}$. We claim that $\adms(G'',k)\leq \adms(G,k)-1$. Let $C=\bigsqcup_{i=1}^r C_i$ be a $k$-admissible set of $G''$ such that 
		\[
		\adms(G'',k)=\ind(G[C]).
		\]
		Then we have
		\[
		k\leq \sum_{i=1}^{r} \beta(G[C_i])\leq r+k-1.
		\]
		Now take $C_{r+1}=\{u_1,v_{1,1}\}$ and set $C'=\bigsqcup_{i=1}^{r+1} C_i$. Then observe that there is no edge among vertices of $G[C_i]$ and $G[C_{r+1}]$ in $G$, for any $1\leq i\leq r$. Moreover, $\beta(G[C_{r+1}])=1$ and so,
		\[
		I(G[C_{r+1}])^{\{\beta(G[C_{r+1}])\}} = I(G[C_{r+1}]) = \l u_1v_{1,1}\r.
		\]
		Finally,
		\[
		\beta(G[C']) = \sum_{i=1}^{r+1} \beta(G[C_i]) = \beta(G[C]) + 1,
		\]
		and hence, $k+1 \leq \sum_{i=1}^{r+1} \beta(G[C_i]) \leq r + k$, which can be rewritten as
		\[
		k \leq \sum_{i=1}^{r+1} \beta(G[C_i]) \leq (r+1) + k - 1.
		\]
		Therefore, $C'$ is a $k$-admissible set of $G$. Note that
		\[
		\ind(G[C']) = |C'| - \beta(G[C']) = |C| + 2 - \beta(G[C]) - 1 = \alpha(G[C]) + 1.
		\]
		This proves that $\adms(G'',k) \leq \adms(G,k) - 1$. The inequality $\adms(G'',k-1) \leq \adms(G,k) - 1$ can be obtained using \Cref{lem: induced adm} and \Cref{lem: ind lemma 2}. Therefore, in this case, we have
		\[
		\reg(I(G')^{\{k\}}:u_1) \leq \adms(G,k)+k - 1,
		\]
		and consequently,
		\[
		\reg(I(G)^{\{k\}}:u_1) \leq \adms(G,k)+k - 1.
		\]
		
		\par
		
		It now remains to consider the case when $d\geq 3$. For this, we first take $W=\{u_1,v_{1,1}\}$. Then, by similar arguments as in the previous case, we have 
		\begin{align*}
			\reg(I(G)^{\{k\}}:u_1)& \leq \max\{\reg(I(G\setminus v_{1,1})^{\{k\}}:u_1),\; \reg(I(G)^{\{k\}}:u_1v_{1,1})+1 \}\\
			&\leq \max\{\reg(I(G\setminus v_{1,1})^{\{k\}}:u_1), \adms(G\setminus \{u_1,v_{1,1},\dots , v_{1,s}\},k-1)+k-1 \}      
		\end{align*}
		Again to find an upper bound for $\reg(I(G\setminus v_{1,1})^{\{k\}}:u_1)$, we consider $W=\{u_1,v_{1,2}\}$ and proceed similarly. After $s$-many steps, we finally obtain, 
		\begin{align*}
			\reg(I(G)^{\{k\}}:u_1) &\leq \max\{\reg(I(G\setminus \{v_{1,1}, \dots , v_{1,s}\})^{\{k\}}:u_1), \\ & \hspace{3cm} \adms(G\setminus \{u_1,v_{1,1},\dots , v_{1,s}\},k-1)+k-1 \}.
		\end{align*} 
		Now, let us set $G'=G\setminus \{v_{1,1}, \dots , v_{1,s}\}$. Note that in $G'$, the block $L'$ with $V(L')=\{u_1,\dots, u_d\}$ is a special block and $d\geq 3$. It follows that $L$ is a special block of Type III in $G'$ with $\mathcal{N}_{G'}(L,u_1)=\emptyset$. Thus, by the induction hypothesis on (2), we have 
		\[
		\reg(I(G')^{\{k\}}:u_1)\leq \adms(G,k)+k-1.
		\]
		So now, it remains to show that 
		\[
		\adms(G \setminus \{u_1, v_{1,1}, \dots, v_{1,s}\}, k-1) \leq \adms(G, k)-1,
		\]
		and it directly follows from \Cref{lem: ind lemma 2}. This completes the proof of (3)
		
		\noindent
		\textbf{Proof of (1):} Let $L$ be a special block of $G$ which is either of Type II, or of Type III. Let $V(L)=\{u_1,\dots , u_d\}$. Then either $\mathcal{N}_G(L,u_i)=\emptyset$ for some $i\in [d-1]$, or $\mathcal{N}_G(L,u_i)\neq \emptyset$ for each $i\in [d-1]$. Thus, using (2) and (3), we can choose some $i\in[d-1]$ such that $\reg(I(G)^{\{k\}}:u_i)\leq \adms(G,k)+k-1$.  On the other hand,
		\begin{align*}
			\reg( I(G)^{\{k\}}+\l u_i\r)&\le  \reg(I(G\setminus u_i)^{\{k\}})\\
			 &\leq \adms(G\setminus u_i,k)+k\\
			 & \le \adms(G,k)+k.
		\end{align*}
		where the inequality in the first line follows from \Cref{regularity lemma} (i), the inequality in the second line follows from the
		  by the induction hypothesis on (1), and the inequality in the last line follows from  \Cref{lem: induced adm}. Hence, the assertion $\reg(I(G)^{\{k\}}) \leq \adms(G,k)+k$ follows from \Cref{regularity lemma} (iii). This completes the proof of (1).
		  
		  Finally, we have $\reg(I(G)^{\{k\}})=\adms(G,k)+k$ using (1) and \Cref{cor:lower bound sqf symbolic}. This completes the proof of the theorem.
	\end{proof}

    \begin{remark}\label{lem: reg of complete}
        Since the complete graph $K_n$ is a block graph, it follows from \Cref{thm: block graph main} that $\reg (I(K_n)^{\{k\}})=\adms(K_n,k)+k$ for all $1\leq k\leq \beta(K_n)=n-1$. However, one can simply observe that $\adms(K_n,k)=1$ for all such $k$. Indeed, if we take $C\subseteq V(K_n)$ with $|C|\ge 2$ such that $C=\sqcup_{i=1}^rC_i$ is a $k$-admissible set of $K_n$, then the induced subgraph on $C$ is a complete subgraph of $K_n$, and hence by Condition (2) of \Cref{def: sym aim for graph}, we must have $r=1$. Also, $E(G[C])\neq \emptyset$ as $|C|\geq 2$. Moreover, condition (3) of \Cref{def: sym aim for graph} is satisfied if and only if  $|C|=k+1$. And finally, if $|C|=k+1$ then  $I(G[C])^{\{\beta(G[C])\}}=I(G[C])^{\{k\}}= (\mathbf{x}_C)$. Therefore, a vertex set $C\subseteq V(K_n)$ is a $k$-admissible set if and only if $|C|=k+1$, and consequently, $\adms(K_n,k)=\max\{\alpha(G[C]): |C|=k+1\}=1$.
    \end{remark}

\section{Second squarefree symbolic power of Cohen-Macaulay Chordal graphs}\label{sec: CM chordal}
In this section, we consider the class of Cohen-Macaulay chordal graphs and show that the general lower bound for the second squarefree symbolic power is attained. Recall the following combinatorial classification of the Cohen-Macaulay property of the edge ideal of a chordal graph by Herzog, Hibi, and Zheng in \cite{CMchordal}. These graphs are simply called the Cohen-Macaulay chordal graph.
	
	\begin{definition}
		A graph $G$ is said to be a \emph{Cohen-Macaulay chordal} graph if $G$ is chordal and $V(G)$ can be partitioned into $W_1,\ldots,W_m$ such that $G[W_i]$ is a maximal clique containing a free vertex for each $i\in[m]$.
	\end{definition}
	
	 To prove the regularity formula, we require some auxiliary lemmas regarding squarefree symbolic powers of edge ideals and their behavior with various colon operations. 
	
	\begin{lemma}\label{lem: symb sq-free 4}
		Let $\{x,y\}\in E(G)$ be such that $N_G[x]\subseteq N_G[y]$. Then for any $1\leq k\leq \beta(G)$
		\[
		(I(G\setminus y)^{\{k\}}:x)\subseteq (I(G\setminus x)^{\{k\}}:y).
		\]
	\end{lemma}
	\begin{proof}
		Let $f$ be a squarefree monomial such that $xf\in I(G\setminus y)^{\{k\}}$. Then for all $P\in \ass(I(G\setminus y))$ we have $|P\cap \supp(xf)|\geq k$. We may assume that $x,y\notin \supp(f)$. To show $f\in (I(G\setminus x)^{\{k\}}:y)$, it is enough to show that $|Q\cap \supp(yf)|\geq k$ for any $Q\in \ass(I(G\setminus x))$. We now consider the following two cases:
		
		\noindent
		\textbf{Case I:} Assume that $y\in Q$. If we consider $Q'=(Q\cup \{x\})\setminus \{y\}$, then observe that $Q'$ is a vertex cover of $G\setminus y$. Since $xf\in I(G\setminus y)^{\{k\}}$, it follows that $|Q'\cap \supp(xf)|\geq k$, which implies that $|Q'\cap \supp(f)|\geq k-1$, and hence $|Q\cap \supp(yf)|\geq k$.
		
		\noindent
		\textbf{Case II:} Assume that $y\notin Q$. Since $N_G[x]\subseteq N_G[y]$, we see that $Q$ is a vertex cover of $G\setminus y$. Again, since $xf\in I(G\setminus y)^{\{k\}}$, we have $|Q\cap \supp(xf)|\geq k$. As $x\notin Q$, we get $|Q\cap \supp(f)|\geq k$, which implies that $|Q\cap \supp(yf)|\geq k$.
		
		\noindent
		Thus, for any $Q\in \ass(I(G\setminus x))$, we have $|Q\cap \supp(yf)|\geq k$. Therefore, $f\in (I(G\setminus x)^{\{k\}}:y)$.
	\end{proof}
	
	If $G$ is any chordal graph, and $\{x,y\}\in E(G)$, then $(I(G)^{[2]}:xy)$ is again a quadratic squarefree monomial ideal, and thus can be realized as an edge ideal of a graph $H$. Indeed, by \cite[Proposition 5.3]{ChauDasRoySahaSqfOrd}, we have $(I(G)^{[2]}:xy)=I(H)$, where $H$ is a weakly chordal graph with the vertex set and edge set described as follows:
		\begin{align*}
		V(H)&=V(G)\setminus\{x,y\}\text{ and } \\ E(H)&=E(G\setminus\{x,y\})\cup\{\{u,v\}\mid u\in N_G(x)\setminus\{y\},v\in N_G(y)\setminus\{x\},\text{ and }u\neq v\}.
	\end{align*}
	 On the other hand, it follows from the explicit expression of the second symbolic powers of edge ideals of simple graphs (see \cite[Corollary 3.12]{Sullivant2008}) that $I(G)^{\{2\}}=I(G)^{[2]}+J_3(G)$, where $J_3(G)$ is the ideal generated by $x_ix_jx_k$, where $x_i,x_j,x_k$ forms a triangle in $G$. Let $\widetilde{G}$ denote the induced subgraph of $H$ on the vertex set $V(H)\setminus\{a\mid \{a,x,y\}\text{ forms a triangle in }G\}$. Then the following lemma follows from the above discussion.
	
	\begin{lemma}\label{lem: CM chordal colon}
		Let $\{x,y\}$ be an edge of a chordal graph $G$. Then 
		\[
		(I(G)^{\{2\}}:xy)=I(\widetilde{G})+\l a\mid a\in V(H)\setminus V(\widetilde{G})\r,
		\]
		where $\widetilde{G}$ is a weakly chordal graph.
	\end{lemma}
	
	We now turn our attention to the relationship between the admissible independence numbers of $G$ and $\widetilde{G}$, when $G$ is a Cohen-Macaulay chordal graph. These relationships will be helpful in the proof of the main theorem in this section.
	
	\begin{lemma}\label{lem: CM chordal aim}
		Let $G$ be a Cohen-Macaulay chordal graph with the decomposition $V(G)=W_1\sqcup W_2\sqcup \dots\sqcup W_m$ such that $m\ge 2$ and for each $i\in [m]$, $G[W_i]\cong K_{t_i}$, where $t_i\ge 2$. 
		\begin{enumerate}
			\item Let $x\in W_i$ such that $W_i\setminus \{x\}$ contains at least one free vertex of $G$. Let $y\in N_G[x]\setminus W_i$ and $(I(G)^{\{2\}}:xy)=I(\widetilde{G})$. Then $\nu_1(\widetilde{G})\le \adms(G,2)-1$.
			
			\item Let $x,y\in W_i$ be free vertices of $G$. Then $(I(G)^{\{2\}}:xy)=I(G\setminus W_i)+( W_i\setminus \{x,y\} )$. Moreover, $\nu_1(G\setminus W_i)\le \adms(G,2)-1$.
		\end{enumerate}  
	\end{lemma}
	\begin{proof} (1) Follows from \cite[Lemma 5.5]{ChauDasRoySahaSqfOrd}, since $\widetilde{G}$ is an induced subgraph of $H$.
		
		\noindent
		(2) The first assertion follows from \Cref{lem: CM chordal colon}. The last assertion follows from the fact that given an induced matching $M$ of $G\setminus W_i$, the set $M\cup \{\{x,y\}\}$ is an induced matching of $G$, and hence a $2$-admissible set of $G$ with $\adms(G[V(M)\cup \{x,y\}])=|M|+1$.
	\end{proof}
	
	\begin{lemma}\label{lem: ind lemma}
		Let $G$ be any graph and $x,y\in V(G)$ such that $N_G[x]\subseteq N_G[y]$. Then  $\adms(G\setminus N_G[y],k)\le\adms(G,k)-1$ for any $1\le k\le \beta(G)$.
	\end{lemma}
	\begin{proof}
		Set $H=G\setminus N_G[y]$. Let $C=\sqcup_{i=1}^rC_i$ be a $k$-admissible set of $H$ such that $\adms(H,k)= \alpha(H[C])$. Now, set $C_{r+1}=\{x,y\}$ and $C'=\sqcup_{i=1}^{r+1}C_i$. Note that $G[C]=H[C]$, and hence there is no edge among vertices of $G[C_{r+1}]$ and $G[C]$, as $N_G[x,y]=N_G[y]$ and $N_G[y]\cap V(H)=\emptyset$. Moreover, $\beta(G[C_{r+1}])=1$ and hence $I(G[C_{r+1}])^{\{1\}}=( xy)$. Thus,
		\[
		k+1\leq \sum_{i=1}^{r+1}\beta(G[C_i])\leq r+k,
		\]
		 which can be rewritten as $k\leq \sum_{i=1}^{r+1}\beta(G[C_i])\leq (r+1)+k-1$. Therefore, $C'$ is a $k$-admissible set of $G$. Note that $\alpha(G[C'])=\alpha(G[C])+\alpha(G[C_{r+1}])=\adms(H,k)+1$, and this completes the~ proof.
	\end{proof}
	
	
	We are now in the position to prove the main theorem of this section. The proof goes along similar lines as \cite[Theorem 5.8]{ChauDasRoySahaSqfOrd}.
	
	\begin{theorem}\label{thm: CMchordal}
		Let $G$ be a Cohen-Macaulay chordal graph with  $\beta(G)\geq 2$. Then 
		\[
		\mathrm{reg}(I(G)^{\{2\}})= \adms(G,2)+2.
		\]
	\end{theorem}
	\begin{proof}
		The inequality $\mathrm{reg}(I(G)^{\{2\}})\ge \adms(G,2)+2$ follows from \Cref{cor:lower bound sqf symbolic}. To prove the upper bound $\mathrm{reg}(I(G)^{\{2\}})\le \adms(G,2)+2$, we use induction on $|V(G)|$. Since $\beta(G)\geq 2$, we have $|V(G)|\ge 3$. If $|V(G)|=3$, then $G=K_3$. Thus, $I(G)^{\{2\}}$ is a principal ideal generated by all the vertices of $G$, and hence, it is easy to see that $\mathrm{reg}(I(G)^{\{2\}})=\adms(G,2)+2$.

		Now, suppose $|V(G)|\ge 4$. Since $G$ is a Cohen-Macaulay chordal graph, we can write $V(G)=W_1\sqcup W_2\sqcup \dots\sqcup W_m$ such that $G[W_i]\cong K_{t_i}$ for each $i\in [m]$, where $t_i\ge 2$. Then we have the following two cases:
		
		\noindent
		\textbf{Case I:} Suppose that $G$ is not a disjoint union of complete graphs. Then there exists some $i\in [m]$ and $x\in W_i$ such that $N_G(x)\setminus W_i\neq\emptyset$. Without loss of generality, let $i=1$. Moreover, let $W_1=\{x_1,\ldots,x_r\}$ for some $r\ge 2$ such that $x_r$ is a free vertex of $G$, and $N_G(x_1)\setminus W_1\neq\emptyset$. Suppose $N_G(x_1)\setminus W_1=\{y_1,\ldots,y_s\}$.
		
		\noindent
		\textbf{Claim 1:} For each $l\in[s-1]\cup\{0\}$, 
		\[
		\mathrm{reg}((I(G)^{\{2\}}+( x_1y_1,\ldots,x_1y_l)):x_1y_{l+1})\le\adms(G,2),
		\]
		where $I(G)^{\{2\}}+( x_1y_1,\ldots,x_1y_l)=I(G)^{\{2\}}$ in case $l=0$.
		
		\noindent
		\textit{Proof of Claim 1}: Let $G_l=G\setminus\{y_1,\ldots,y_{l}\}$, where $G_0=G$. Then, it is easy to see that $G_l$ is again a Cohen-Macaulay chordal graph. Now,
		\begin{align*}
			\reg((I(G)^{\{2\}}+\l x_1y_1,\ldots,x_1y_l\r):x_1y_{l+1})&=\reg((I(G_l)^{\{2\}}:x_1y_{l+1})+\l y_1,\ldots,y_l\r)\\
			&=\reg(I(\widetilde{G_l}))\\		
&			= \nu_1(\widetilde{G_l})+1\\
			&\le \adms(G_l,2)\\
			 &\le \adms(G,2),
		\end{align*}
		where $\widetilde{G_l}$ is a weakly chordal graph (by \Cref{lem: CM chordal colon}); the equality in the third line follows from \cite[Theorem 14]{Woodroofe2014}, the inequality in the fourth line follows from \Cref{lem: CM chordal aim}, and the inequality in the fifth line follows from \Cref{lem: induced adm}. This completes the proof of Claim~1.
		
		\noindent
		\textbf{Claim 2:} For each $p\in[r-1]$, 
		\[
		\mathrm{reg}((I(G)^{\{2\}}+( x_1y_1,\ldots,x_1y_s,x_1x_2,\ldots,x_1x_p)):x_1x_{p+1})\le \adms(G,2),
		\]
		where $I(G)^{\{2\}}+( x_1y_1,\ldots,x_1y_s,x_1x_2,\ldots,x_1x_p)=I(G)^{\{2\}}+( x_1y_1,\ldots,x_1y_s)$ in case $p=1$.
		
		\noindent
		\textit{Proof of Claim 2}: Let $H_p=G\setminus\{y_1,\ldots,y_{s},x_2,\ldots,x_p\}$, where $H_1=G\setminus\{y_1,\ldots,y_{s}\}$. Then again, it is easy to see that $H_p$ is a Cohen-Macaulay chordal graph. Moreover, both $x_{1}$ and $x_{r}$ are free vertices of $H_p$. Now,
		\begin{align*}
			\mathrm{reg}((I(G)^{\{2\}}+( x_1y_1,&\ldots,x_1y_s,x_1x_2,\ldots,x_1x_p)):x_1x_{p+1})\\& =\mathrm{reg}((I(H_p)^{\{2\}}:x_1x_{p+1})+( y_1,\ldots,y_s,x_2,\ldots,x_p))\\
			& =\reg(I(H_p\setminus \{x_1,x_{p+1},x_{p+2},\dots x_r\}))\\
			&= \nu_1(H_p\setminus \{x_1,x_{p+1},x_{p+2},\dots x_r\})+1 \\
			&\le  \adms(H_p,2)\\
			&
			\le \adms(G,2),
		\end{align*}
		where the equality in the second line and the inequality in the fourth line follow from \Cref{lem: CM chordal aim}, the equality in the third line follows from \cite[Theorem 14]{Woodroofe2014},  and the last inequality follows from \Cref{lem: induced adm}. This completes the proof of Claim 2.
		
		Let $\Gamma=G\setminus N_G[x_1]$. Then $\Gamma$ is a Cohen-Macaulay chordal graph. Moreover,
		\begin{equation}\label{eq31}
			\begin{split}
				\reg((I(G)^{\{2\}}+( x_1y_1,\ldots,x_1y_s,x_1x_2,\ldots,x_1x_r)):x_1)&=\reg(I(\Gamma)^{\{2\}}+\l y_1,\ldots,y_s,x_2,\ldots,x_r\r)\\
				&\le \adms(\Gamma,2)+2\\
				&\le \adms(G,2)+1,
			\end{split}
		\end{equation}
		where the equality follows from \cite[Lemma 2.9]{DRS20242}, the first inequality follows from the induction hypothesis, and the second inequality follows from \Cref{lem: ind lemma}.
		
		Next, we have
		\begin{equation}\label{eq32}
			\begin{split}
				\reg(I(G)^{\{2\}}+\l x_1y_1,\ldots,x_1y_s,x_1x_2,\ldots,x_1x_r,x_1\r)&=\reg(I(G\setminus x_1)^{\{2\}})\\
				&\le\adms(G\setminus x_1,2)+2\\
				&\le \adms(G,2)+2,
			\end{split}
		\end{equation}
		where the first and second inequalities follow from the induction hypothesis and \Cref{lem: induced adm}, respectively. Now, using the regularity lemma (\Cref{regularity lemma}) along with inequalities~(\ref{eq31}) and (\ref{eq32}), we get 
		\[\reg(I(G)^{\{2\}}+\l x_1y_1,\ldots,x_1y_s,x_1x_2,\ldots,x_1x_r\r)\le\adms(G,2)+2.\] 
		Applying the regularity lemma with the above inequality and the inequality proved in Claim 2, we get \[\reg(I(G)^{\{2\}}+\l  x_1y_1,\ldots,x_1y_s,x_1x_2,\ldots,x_1x_{r-1}\r)\le\adms(G,2)+2.\] 
		By a repeated use of Claim 2 and the regularity lemma, we get $\reg(I(G)^{\{2\}}+\l x_1y_1,\ldots,x_1y_s\r)\le\adms(G,2)+2$. Now, from Claim 1 we have $\reg((I(G)^{\{2\}}+\l x_1y_1,\ldots,x_1y_{s-1}\r):x_1y_{s})\le\adms(G,2)$. Thus using the regularity lemma again we obtain $\reg((I(G)^{\{2\}}+\l x_1y_1,\ldots,x_1y_{s-1}\r)\le\adms(G,2)+2$. By a repeated use of Claim 1 and the regularity lemma this time, we finally get $\reg(I(G)^{\{2\}}\le\adms(G,2)+2$, as desired.
		
		\noindent
		\textbf{Case II:} $G$ is a disjoint union of complete graphs. In this case, let $W_1=\{x_1,\ldots,x_r\}$ for some $r\ge 2$. Then, proceeding as in Case I above, we obtain the following:
		\begin{enumerate}[label=(\roman*)]
			\item For each $p\in[r-1]$, $\mathrm{reg}((I(G)^{\{2\}}+( x_1x_2,\ldots,x_1x_p)):x_1x_{p+1})\le\adms(G,2)$, where $I(G)^{\{2\}}+( x_1x_2,\ldots,x_1x_p)=I(G)^{\{2\}}$ in case $p=1$;
			
			\item $ \mathrm{reg}((I(G)^{\{2\}}+( x_1x_2,\ldots,x_1x_r)):x_1)\le\adms(G,2)+1$;
			
			\item $ \mathrm{reg}((I(G)^{\{2\}}+( x_1x_2,\ldots,x_1x_r,x_1))\le\adms(G,2)+2$.
		\end{enumerate}
		Now, using the regularity lemma along with (ii) and (iii) above, we have $\mathrm{reg}(I(G)^{\{2\}}+( x_1x_2,\ldots,x_1x_r))\le \adms(G,2)+2$. Thus, as before, by a repeated use of the regularity lemma and (i) above, we obtain $\reg(I(G)^{\{2\}})\le\adms(G,2)+2$. This completes the proof of the theorem.
	\end{proof}

    \begin{remark}
As established in \Cref{thm: block graph main} and \Cref{thm: CMchordal}, one has 
an exact formula for $\reg(I(G)^{\{k\}})$ for all $k$ when $G$ is a block graph, 
and for $\reg(I(G)^{\{2\}})$ when $G$ is a Cohen--Macaulay chordal graph. These 
constitute significant subclasses of chordal graphs. It is therefore natural to 
ask whether the equality 
\[
\reg(I(G)^{\{k\}}) \;=\; \adms(G,k)+k
\]
holds for all chordal graphs and for all $1 \leq k \leq \nu(G)$. Our computations using Macaulay2 \cite{M2} verify that this formula is indeed valid for all chordal graphs up to $6$ vertices.

    \end{remark}

    \section*{Acknowledgements}

The authors thank Nursel Erey, S. A. Seyed Fakhari, and Takayuki Hibi for some helpful comments. Chau, Das, and Roy are supported by Postdoctoral Fellowships at the Chennai Mathematical Institute and a grant from the Infosys Foundation.

\subsection*{Data availability statement} Data sharing does not apply to this article as no new data were created or analyzed in this study.

 \subsection*{Conflict of interest} The authors declare that they have no known competing financial interests or personal relationships that could have appeared to influence the work reported in this paper.

\bibliographystyle{abbrv}
\bibliography{ref}

\begin{thebibliography}{10}

\bibitem{BHZN}
M.~Bigdeli, J.~Herzog, and R.~Zaare-Nahandi.
\newblock On the index of powers of edge ideals.
\newblock {\em Comm. Algebra}, 46(3):1080--1095, 2018.

\bibitem{BM76}
J.~A. Bondy and U.~S.~R. Murty.
\newblock {\em Graph theory with applications}.
\newblock American Elsevier Publishing Co., Inc., New York, 1976.

\bibitem{Brodmann1979}
M.~Brodmann.
\newblock Asymptotic stability of {${\rm Ass}(M/I\sp{n}M)$}.
\newblock {\em Proc. Amer. Math. Soc.}, 74(1):16--18, 1979.

\bibitem{ChauDasRoySahaSqfOrd}
T.~Chau, K.~K. Das, A.~Roy, and K.~Saha.
\newblock Admissible matchings and the {C}astelnuovo-{M}umford regularity of square-free powers.
\newblock {\em arXiv:$2504.11941$}, 2025.

\bibitem{CEM25}
T.~Chau, N.~Erey, and A.~Maithani.
\newblock The {S}carf complex of squarefree powers, symbolic powers of edge ideals, and cover ideals of graphs.
\newblock {\em arXiv:$2503.13337$}, 2025.

\bibitem{CFL1}
M.~Crupi, A.~Ficarra, and E.~Lax.
\newblock Matchings, square-free powers and betti splittings.
\newblock {\em Illinois J. Math}, 69(2):353--372, 2025.

\bibitem{CutkoskyHerzogTrung1999}
S.~D. Cutkosky, J.~Herzog, and N.~V. Trung.
\newblock Asymptotic behaviour of the {C}astelnuovo-{M}umford regularity.
\newblock {\em Compositio Math.}, 118(3):243--261, 1999.

\bibitem{DHS}
H.~Dao, C.~Huneke, and J.~Schweig.
\newblock Bounds on the regularity and projective dimension of ideals associated to graphs.
\newblock {\em J. Algebraic Combin.}, 38(1):37--55, 2013.

\bibitem{DRS20242}
K.~K. Das, A.~Roy, and K.~Saha.
\newblock Square-free powers of {C}ohen-{M}acaulay forests, cycles, and whiskered cycles.
\newblock {\em arXiv:$2409.06021$}, 2024.

\bibitem{CMSimplicialForests}
K.~K. Das, A.~Roy, and K.~Saha.
\newblock Square-free powers of {C}ohen-{M}acaulay simplicial forests.
\newblock {\em arXiv:$2502.18396$, to appear in Proc. Amer. Math. Soc.}, 2025.

\bibitem{DungHienNguyenTrung2021}
L.~X. Dung, T.~T. Hien, H.~D. Nguyen, and T.~N. Trung.
\newblock Regularity and {K}oszul property of symbolic powers of monomial ideals.
\newblock {\em Math. Z.}, 298(3-4):1487--1522, 2021.

\bibitem{Edmonds}
J.~Edmonds.
\newblock An introduction to matching.
\newblock Available at \url{https://web.eecs.umich.edu/~pettie/matching/Edmonds-notes.pdf}, 1967.

\bibitem{EHHS}
N.~Erey, J.~Herzog, T.~Hibi, and S.~Saeedi~Madani.
\newblock Matchings and squarefree powers of edge ideals.
\newblock {\em J. Combin. Theory Ser. A}, 188:Paper No. 105585, 24, 2022.

\bibitem{ErHi1}
N.~Erey and T.~Hibi.
\newblock Squarefree powers of edge ideals of forests.
\newblock {\em Electron. J. Combin.}, 28(2):Paper No. 2.32, 16, 2021.

\bibitem{FiHeHi}
A.~Ficarra, J.~Herzog, and T.~Hibi.
\newblock Behaviour of the normalized depth function.
\newblock {\em Electron. J. Combin.}, 30(2):Paper No. 2.31, 16, 2023.

\bibitem{ficarraCM2024}
A.~Ficarra and S.~Moradi.
\newblock Monomial ideals whose all matching powers are {C}ohen-{M}acaulay.
\newblock {\em arXiv:$2410.01666$}, 2024.

\bibitem{FranciscoHaVanTuyl2009}
C.~A. Francisco, H.~T. H\`a, and A.~Van~Tuyl.
\newblock Splittings of monomial ideals.
\newblock {\em Proc. Amer. Math. Soc.}, 137(10):3271--3282, 2009.

\bibitem{M2}
D.~R. Grayson and M.~E. Stillman.
\newblock Macaulay2, a software system for research in algebraic geometry.
\newblock Available at \url{http://www2.macaulay2.com}.

\bibitem{THT2020}
H.~T. H\`a, H.~D. Nguyen, N.~V. Trung, and T.~N. Trung.
\newblock Symbolic powers of sums of ideals.
\newblock {\em Math. Z.}, 294(3-4):1499--1520, 2020.

\bibitem{HaVanTuyl2008}
H.~T. H\`a and A.~Van~Tuyl.
\newblock Monomial ideals, edge ideals of hypergraphs, and their graded {B}etti numbers.
\newblock {\em J. Algebraic Combin.}, 27(2):215--245, 2008.

\bibitem{HHBook}
J.~Herzog and T.~Hibi.
\newblock {\em Monomial ideals}, volume 260 of {\em Graduate Texts in Mathematics}.
\newblock Springer-Verlag London, Ltd., London, 2011.

\bibitem{HerzogHibiTrung(VertexCoverAlgebras)2007}
J.~Herzog, T.~Hibi, and N.~V. Trung.
\newblock Symbolic powers of monomial ideals and vertex cover algebras.
\newblock {\em Adv. Math.}, 210(1):304--322, 2007.

\bibitem{HHZ04}
J.~Herzog, T.~Hibi, and X.~Zheng.
\newblock Monomial ideals whose powers have a linear resolution.
\newblock {\em Math. Scand.}, 95(1):23--32, 2004.

\bibitem{CMchordal}
J.~Herzog, T.~Hibi, and X.~Zheng.
\newblock Cohen-{M}acaulay chordal graphs.
\newblock {\em J. Combin. Theory Ser. A}, 113(5):911--916, 2006.

\bibitem{HoaTrung2007}
L.~T. Hoa and T.~N. Trung.
\newblock Partial {C}astelnuovo-{M}umford regularities of sums and intersections of powers of monomial ideals.
\newblock {\em Math. Proc. Cambridge Philos. Soc.}, 149(2):229--246, 2010.

\bibitem{KaNaQu}
E.~Kamberi, F.~Navarra, and A.~A. Qureshi.
\newblock On squarefree powers of simplicial trees.
\newblock {\em arXiv:$2406.13670$}, 2024.

\bibitem{Kodiyalam1999}
V.~Kodiyalam.
\newblock Asymptotic behaviour of {C}astelnuovo-{M}umford regularity.
\newblock {\em Proc. Amer. Math. Soc.}, 128(2):407--411, 1999.

\bibitem{MinhEtAl(SmallPowers)2022}
N.~C. Minh, L.~D. Nam, T.~D. Phong, P.~T. Thuy, and T.~Vu.
\newblock Comparison between regularity of small symbolic powers and ordinary powers of an edge ideal.
\newblock {\em J. Combin. Theory Ser. A}, 190:Paper No. 105621, 30, 2022.

\bibitem{MinhVu(SmallPowers)2024}
N.~C. Minh and T.~Vu.
\newblock A characterization of graphs whose small powers of their edge ideals have a linear free resolution.
\newblock {\em Combinatorica}, 44(2):337--353, 2024.

\bibitem{HV19}
H.~D. Nguyen and T.~Vu.
\newblock Powers of sums and their homological invariants.
\newblock {\em J. Pure Appl. Algebra}, 223(7):3081--3111, 2019.

\bibitem{Fakhari2024}
S.~A. {Seyed Fakhari}.
\newblock {On the Castelnuovo–Mumford regularity of squarefree powers of edge ideals}.
\newblock {\em Journal of Pure and Applied Algebra}, 228(3):107488, 2024.

\bibitem{FakhariSymbSq-free}
S.~A. Seyed~Fakhari.
\newblock On the regularity of squarefree part of symbolic powers of edge ideals.
\newblock {\em J. Algebra}, 665:103--130, 2025.

\bibitem{Sullivant2008}
S.~Sullivant.
\newblock Combinatorial symbolic powers.
\newblock {\em J. Algebra}, 319(1):115--142, 2008.

\bibitem{RHV}
R.~H. Villarreal.
\newblock {\em Monomial algebras}.
\newblock Monographs and Research Notes in Mathematics. CRC Press, Boca Raton, FL, second edition, 2015.

\bibitem{Woodroofe2014}
R.~Woodroofe.
\newblock Matchings, coverings, and {C}astelnuovo-{M}umford regularity.
\newblock {\em J. Commut. Algebra}, 6(2):287--304, 2014.

\end{thebibliography}
\end{document}